\documentclass[11pt,reqno]{amsart}
\usepackage{mathtools}
\usepackage[table]{xcolor}
\usepackage{amssymb}
\usepackage{enumitem}
\usepackage[margin=1in]{geometry}
\usepackage[utf8]{inputenc}
\usepackage[cyr]{aeguill}
\PassOptionsToPackage{hyphens}{url}
\usepackage[pdfdisplaydoctitle=true,
            colorlinks=true,
            urlcolor=blue,
            citecolor=blue,
            linkcolor=blue,
            pdfstartview=FitH,
            pdfpagemode= UseNone,
            bookmarksnumbered=true]{hyperref}
\usepackage{todonotes}
\usepackage[all]{xy}
\usepackage{stmaryrd}
\usepackage{appendix}
\usepackage{enumitem}
\usepackage{multirow}
\usepackage[normalem]{ulem}

\newtheorem{theorem}{Theorem}[section]
\newtheorem{corollary}[theorem]{Corollary}
\newtheorem{lemma}[theorem]{Lemma}
\newtheorem{proposition}[theorem]{Proposition}
\theoremstyle{definition}
\newtheorem{definition}[theorem]{Definition}

\newtheorem{example}[theorem]{Example}
\newtheorem{remark}[theorem]{Remark}

\newcommand{\R}{\mathbb R}

\newcommand{\dx}{\lambda}
\newcommand{\Prob}{\operatorname{Prob}(M)}

\newcommand{\Dens}{\operatorname{Dens}_+(M)}

\newcommand{\Diff}{\operatorname{Diff}(M)}
\newcommand{\DiffRshift}{\widetilde{\operatorname{Diff}}(\mathbb R)}
\newcommand{\DiffRLie}{\mathfrak g}
\newcommand{\DiffR}{\operatorname{Diff}(\mathbb R)}
\newcommand{\DensR}{\widetilde{\operatorname{Dens}}_{+}(\R)}

\newcommand{\SDiff}{\operatorname{Diff}_{\dx}(M)}

\newcommand{\GFR}{G}

\newcommand{\nablabar}{\overline{\nabla}}

\newcommand{\na}{\nabla^{(\alpha)}}
\newcommand{\nabar}{\overline{\nabla}^{(\alpha)}}
\newcommand{\KL}{\mathrm{KL}}

\title[Amari--\v{C}encov $\alpha$-connections and Proudman--Johnson equations]{A Riemannian viewpoint on the Amari--\v{C}encov $\alpha$-connections and Proudman--Johnson equations}

\author[]
{Martin Bauer$^{*}$, Alice Le Brigant, Cy Maor}
\thanks{$^{*}$Corresponding Author.}
\address{M. Bauer: Florida State University; A. Le Brigant: SAMM, Universit\'e Paris 1; C. Maor: Einstein Institute of Mathematics, The Hebrew University of
Jerusalem}
\email{bauer@math.fsu.edu, alice.le-brigant@univ-paris1.fr,
cy.maor@mail.huji.ac.il}

\date{\today}
\subjclass[2010]{%
}

\begin{document}

\begin{abstract}
We give a new geometric interpretation of the Amari--\v{C}encov $\alpha$-connections $\nabla^{(\alpha)}$ from information geometry:
On the space of densities $\Dens$, we show that there exist Riemannian metrics $G^\alpha$, which we call $\alpha$-Fisher--Rao metrics, whose Levi-Civita connections are $\nabla^{(\alpha)}$.
With the exception of $\alpha=0$ (the Fisher--Rao metric), these metrics are non-invariant to the action of the diffeomorphism group $\Diff$, even though the connections are invariant. 
This gives a new way of interpreting the geodesics of the $\nabla^{(\alpha)}$ as energy-minimizing curves.
On the space of probability densities $\Prob$, we show that the same phenomenon holds for $\alpha\in \{-1,0,1\}$ and that the $\alpha$-connections are not metric otherwise. We show that $\nabla^{(\alpha)}$-geodesics on this space can be interpreted as radial projections of straight lines on appropriate hyper-surfaces, and use this geometric picture to obtain geodesic convexity for any $\alpha\in \R$.
In addition, we prove analogous results for appropriate metrics and connections on $\Diff$, which, for the case $M=\R$, imply that the generalized Proudman--Johnson equations on the real line are the Euler--Arnold-type equations of non-right invariant metrics. Finally, in the finite-dimensional case, we show that $\nabla^{(\alpha)}$ can be metric or non-metric depending on the statistical model considered. 
\end{abstract}

\maketitle
\setcounter{tocdepth}{1}
\tableofcontents

\section{Introduction and main results}

Information geometry is concerned with the study of spaces of probability densities as differentiable manifolds. 
Two of the most important structures in information geometry are the Fisher--Rao metric \cite{rao1945,friedrich1991} and the Amari--\v{C}encov $\alpha$-connections \cite{cencov1982,amari2000methods,gibilisco1998connections,newton2012infinite,gibilisco2020p}, for $\alpha\in \mathbb{R}$ (the case $\alpha=0$ corresponding to the Levi-Civita connection of the Fisher--Rao metric). When defined on a parametric family of probability measures, both can be characterized by their invariance with respect to sufficient statistics, as was shown by \v{C}encov on finite sample spaces~\cite{cencov1982} and Ay et al.\ on infinite sample spaces~\cite{ay2015information}. In the non-parametric setting, both the Fisher--Rao metric and the $\alpha$-connections are invariant with respect to diffeomorphic change of the support, and this invariance has been shown to characterize the Fisher-Rao metric~\cite{bauer2016uniqueness}\footnote{The uniqueness of the $\alpha$-connections with respect to this invariance is still, to our knowledge, an open question.}.

These invariant structures have numerous applications in various fields, ranging from statistical inference and learning \cite{amari2016information} to, more recently, optimal transport \cite{ay2024information}, quantum field theory \cite{naudts2024duality}, Markov chains \cite{wolfer2023information}, image analysis~\cite{bauer2015diffeomorphic} and machine learning: generative neural networks \cite{cheng2025alpha}, natural language processing \cite{volpi2021natural} and reinforcement learning \cite{nedergaard2025information}. 

Another motivation to study infinite-dimensional information geometry stems from its link to geometric hydrodynamics~\cite{khesin2024information,khesin2013geometry,lenells2014amari}, where fluid motion is described by geodesics on the space of (volume preserving) diffeomorphisms of the fluid domain. Indeed, given a compact manifold $M$ with a volume form $\lambda$, the map $\varphi\mapsto \varphi_*\lambda$ maps the diffeomorphism group $\Diff$ to the space $\Prob$ of probability densities, thus indentifying $\Prob$ with the quotient space $\Diff/\SDiff$ of diffeomorphisms modulo diffeomorphisms that preserves a reference volume form $\lambda$. 
(In non-compact cases like $M=\R$, this yields the space $\Dens$ of all positive densities with appropriate decay conditions.) 
Thus, one can pullback the Fisher--Rao metric and the $\alpha$-connections to obtain right-invariant metrics and connections on $\Diff$. This was studied in \cite{khesin2013geometry,modin2015generalized} for the Fisher--Rao metric, and in \cite{lenells2014amari,bauer2024p} for the $\alpha$-connections. See \cite{khesin2024information} for a recent overview. Interestingly, for $M=S^1$ and $M=\R$, the geodesic equations that arise on $\Diff$ are well-known equations from hydrodynamics: the Hunter--Saxton equation is the geodesic equation of the Fisher--Rao metric \cite{khesin2013geometry}, and the generalized Proudman--Johnson equations are the geodesic equations of the $\alpha$-connections \cite{lenells2014amari}. The geodesic equations for a general manifold $M$ are thus the higher-dimensional version of these equations, and the study of the geometry of the $\alpha$-connections is a way to gain insight on their solutions.

Recently, the authors of the present paper took a step in that direction by studying the geometry of the $\alpha$-connections on two different spaces: the space of positive densities $\Dens$ and the space of probability densities $\Prob$. They showed in~\cite{bauer2024p} that on $\Dens$, the $\alpha$-connections for $\alpha\in(-1,1)$ define the same geodesics as a family of invariant Finsler metrics, the $L^p$-Fisher--Rao metrics, where $p=2/(1-\alpha)$. On $\Prob$, it was shown that the $\alpha$-connections are the pullback by the \emph{$p$-root map }$\mu\mapsto (\mu/\lambda)^{1/p}$ of the connection on the $L^p$-sphere in the space of functions, obtained by a radial projection of the trivial connection. In both cases, this geometric perspective led to explicit solutions of the geodesic equations: On $\Dens$, these are pullbacks by the $p$-root map  of straight lines, and on $\Prob$, they are pullbacks of projections of straight lines on the $L^p$-sphere. This is analogous to the analysis of the non-periodic \cite{bauer2014homogeneous} and periodic \cite{lenells2007hunter} Hunter--Saxton equation, respectively. On $\DiffR$ this analysis led to solutions of the non-periodic generalized Proudman--Johnson equation (or equivalently, the $r$-Hunter--Saxton equation) \cite{cotter2020r, bauer2022geometric}, and for a general closed manifold $M$, similar analysis was used to estimate the diameter of $\Diff$ with respect to right-invariant Sobolev metrics \cite{bauer2021can}. 

In this paper, we further explore the geometry associated to the $\alpha$-connections on spaces of densities and probability densities, this time from a Riemannian perspective, and its implications for some well-known equations in hydrodynamics.

\subsection{Main results and future directions}

We now describe the main contributions of this article, as well as open questions that arise from our results:
\begin{enumerate}
\item 
    On the space of positive densities $\Dens$, we present a family of Riemannian metrics $G^\alpha$, given by
    \begin{align*}
    G^{\alpha}_{\mu}(a,b)=\int \left(\frac{\mu}{\dx}\right)^{-\alpha-1} \frac{a}{\dx}\frac{b}{\dx}\dx\;,\qquad \mu\in \Dens, a,b\in T_{\mu}\Dens,
    \end{align*}
    where $\frac{a}{\mu}$ denotes the Radon-Nikodym derivative of $a$ with respect to $\mu$. We refer to these Riemannian metrics as the \emph{$\alpha$-Fisher--Rao metrics}, and show that for every $\alpha\in \R$, the $\alpha$-connection $\nabla^{(\alpha)}$ is the Levi-Civita connection of $G^\alpha$ (Theorem~\ref{thm:alphaFR_connection_dens}). With the exception of $\alpha=0$, corresponding to the Fisher--Rao metric, all these metrics are non-invariant to the action of $\Diff$, but their connections are. Note that this characterization of the $\alpha$-connections as a Levi-Civita connection holds for any $\alpha\in \R$, unlike the one of being associated with invariant Finsler metrics~\cite{bauer2024p}, which only holds for $\alpha\in (-1,1)$.\footnote{The $\alpha$-connections are usually defined in the literature for $\alpha\in[-1,1]$ because in that setting, they can be linked to the $\alpha$-divergences (see Lemmas~\ref{lem:alphacon_dens} and \ref{lem:alphacon_prob}). However, in particular in light of the above described link to the  generalized Proudman–Johnson equations, which are defined for any $\alpha\in\mathbb R$,  it is natural to also extend the definition of the $\alpha$-connections to any $\alpha\in\R$, which is what we do in this paper.}
\item 
    Applying this characterization to the case $M=\R$, and pulling it back to $\DiffR$ (with appropriate decay conditions at infinity), we obtain that the non-periodic, generalized Proudman--Johnson equations
    \begin{equation*}
    u_{txx} + (2-\alpha)u_xu_{xx} +uu_{xxx} = 0, 
    \end{equation*}
    are Euler--Arnold-type equations of non-invariant metrics. To the best of our knowledge, the only other case of an invariant equation that arises as geodesic equation of a non-invariant metric on a group of diffeomorphisms is the inviscid Burgers equation, which is the geodesic equation of the non-invariant $L^2$-metric related to optimal mass transport~\cite{otto2001geometry}; see also~\cite{ebin1970groups} where this situation is discussed in greater detail. In future work it would be of interest to obtain a complete characterization of all instances for this phenomenon, both on the groups of diffeomorphisms and on general finite and infinite dimensional Lie groups. 
\item 
    On the space of probability densities $\Prob$, we show that the Levi-Civita connection of the restriction of  $G^\alpha$ to $\Prob$ coincides with the $\alpha$-connection  on $\Prob$ only for $\alpha\in \{0,-1\}$.
    For $\alpha=1$, the $\alpha$-connection is the Levi-Civita connection of a variant of the $1$-Fisher--Rao metric, and for $\alpha\notin \{-1,0,1\}$, it is not a metric connection with respect to any Riemannian metric (Theorem~\ref{thm:alpha_nonmetric}).
    For finite-dimensional parametric statistical models, we show in Proposition~\ref{prop:exponential} that a similar phenomenon holds for two-dimensional exponential families; it would be interesting to know whether this result holds in any dimension. We also show that other statistical models can present a different behavior, namely we 
    give an  example of a statistical model such that all $\alpha$-connections are metric (Example~\ref{ex:finite-dim}).
\item 
    Finally, we show that the geometric picture presented in \cite{bauer2024p} of the geodesics of the $\alpha$-connections  on $\Prob$, as pullbacks of projections of straight lines on an appropriate surface, holds for any $\alpha\in \R$ (Theorem~\ref{thm:sulotion_prob}).
    The difference from the case $\alpha\in (-1,1)$ presented in \cite{bauer2024p} is that the surfaces projected upon are not parts of spheres but can be non-convex and unbounded. 
    We use this geometric picture to show that, for all $\alpha$, $\Prob$ is uniquely geodesically convex with respect to $\nabla^{(\alpha)}$, i.e., that for any two elements in the space there exists a unique $\alpha$-geodesic connecting them.
    In the case of $M=S^1$, which corresponds to the periodic generalized Proudman--Johnson equations, this geometric picture might lead to a geometric explanation of the different blowup behavior of the equations for various values of $\alpha$ (see \cite{sarria2013blow,kogelbauer2020global}), as well as to establish similar results for the higher-dimensional case; as this paper focuses on the geometric picture, we intend to investigate this direction in a separate future work.
\end{enumerate}
A summary of the metric properties of the $\alpha$-connection can be found in Figure~\ref{fig:summary}.
\begin{figure}\label{fig:summary}
\bgroup
\def\arraystretch{1.25}
\begin{tabular}{|c|c|c|c|c|c|c|c|}
  \hline
  $\alpha$                  & $(-\infty,-1)$ & $-1$ & $(-1,0)$ & $0$ & $(0,1)$ & $1$           & $(1,+\infty)$\\
  \hline
  \multirow{2}{*}{$\Dens$} & \cellcolor{gray!50} & \multicolumn{5}{c|}{$L^p$-FR (Finsler, invariant)} & \cellcolor{gray!50}\\
  \cline{2-8}
                           & \multicolumn{7}{c|}{$\alpha$-FR (Riemannian, non invariant)} \\
  \hline
  $\Prob$ & \,\,non metric\,\, & $\alpha$-FR & \,\,non metric\,\, & FR & \,\,non metric\,\, & \,\,\,$\approx \alpha$-FR\,\,\, & \,\,non metric\,\,\\
  \hline
\end{tabular}
\egroup
\caption{Summary of the metric properties of the $\alpha$-connection. Depending on $\alpha$ and whether it is defined on $\Prob$ or $\Dens$, the $\alpha$-connection can either be the Levi-Civita connection of the non-invariant, Riemannian, $\alpha$-Fisher--Rao ($\alpha$-FR) metric, and/or define the same geodesics as the invariant, Finsler, $L^p$-Fisher--Rao ($L^p$-FR) metric for $p=1/(2-\alpha)$, or be non-metric. 
The case $\alpha=1$ on $\Prob$ involves a non-invariant Riemannian metric, which is slightly different from the $\alpha$-Fisher--Rao metric on $\Dens$.} 
\end{figure}

\subsection*{Structure of the paper}

Section~\ref{sec:background} gives some background on the spaces of densities considered in the paper, as well as the Fisher--Rao metric and the $\alpha$-connections. Section~\ref{sec:dens} is devoted to the case of positive densities: we introduce the $\alpha$-Fisher--Rao metrics, use them to give a Riemannian interpretation of the $\alpha$-connections on $\Dens$,  and apply this to the Proudman--Johnson equations. Finally, Section~\ref{sec:prob} concerns the case of probability densities: we study the Riemannian property of the $\alpha$-connections on $\Prob$, their geodesic convexity, and consider the finite-dimensional setting.

\subsection*{Acknowledgments}

This work originates from a question asked by Klas Modin during the program ``Infinite-dimensional Geometry: Theory and Applications'' at the Erwin Schrödinger International Institute for Mathematics and Physics (ESI), which MB and ALB attended. The authors sincerely thank the organizers of the program and the ESI for fostering a stimulating environment that encouraged fruitful discussions and enabled this work. Furthermore, we want to express our gratitude to both Klas Modin and Steve Preston, who provided valuable insights and comments during the process of writing this manuscript. 
The authors also thank Robert Bryant, whose post on MathOverflow\footnote{\url{https://mathoverflow.net/questions/54434/when-can-a-connection-induce-a-riemannian-metric-for-which-it-is-the-levi-civita}} inspired part of the proof of Theorem~\ref{thm:alpha_nonmetric}. In addition we thank the anonymous referees for their careful reading of our manuscript and for their helpful suggestions.

This work was written when CM was visiting the University of Toronto and the Fields Institute; CM is grateful
for their hospitality. 
CM was partially supported by ISF grant 2304/24.
MB and CM were partially supported by BSF grant 2022076. MB was partially supported by NSF Grant CISE-2426549.

\section{Background: Spaces of densities, the Fisher--Rao metric and Amari--\v{C}encov connections}\label{sec:background}

\subsection{The space of densities as an infinite dimensional manifold}

In this section we will introduce several manifolds of mappings, that will play a central role in this article; for a detailed introduction to infinite dimensional Fr\'echet manifolds we refer the interested reader to~\cite{kriegl_michor_1997,hamilton1982inverse,Mao22}.

Throughout this article, unless otherwise mentioned, we let $M$ be a closed (i.e., compact and without boundary), oriented, $n$-dimensional manifold. 
We denote by $\Dens$ the space of smooth positive densities and by $\Prob$ the subspace of smooth probability densities, i.e., 
\begin{align*}
\Dens&:=\{\mu\in \Omega^n(M) : \mu>0\}\\
\Prob&:=\{\mu\in \Dens : \int\mu=1\}.
\end{align*}
Here, $\Omega^n(M)$ is the space of all $n$-forms on $M$, and the positivity means the positivity of the Radon--Nikodym derivative $\frac{\mu}{\lambda}$ of $\mu$ with respect to some volume form $\lambda$ on $M$ (say, of some Riemannian metric).
Since $\Dens$ is an open subset of the Fr\'echet space $\Omega^n(M)$ it carries the structure of a Fr\'echet manifold with tangent space $T_{\mu}\operatorname{Dens}(M)=\Omega^n(M)$.
Similarly, as an affine subspace of a Fr\'echet manifold,  the space of probability densities is a Fr\'echet manifold, where the tangent space is given by 
\begin{align*}
T_{\mu}\Prob=\left\{a \in \Omega^n(M):\int a=0\right\}.
\end{align*}

On both the space of  densities and probability densities
we can consider the pushforward action of the diffeomorphism group $\Diff$.  On $\Dens$ it is given by
\begin{equation}\label{eq:pushforward}
\Diff\times \Dens\ni (\varphi,\mu)\mapsto \varphi_*\mu \in \Dens
\end{equation}
and, since the pushforward by a diffeomorphism is volume preserving, this action restricts to an action on the space of probability densities. 
By a result of Moser~\cite{moser1965volume} the action on the space of probability densities is transitive, which allows us to identify the space of probability densities with the quotient
\begin{equation*}
\Prob\equiv \Diff/\SDiff, 
\end{equation*}
where $\SDiff$ is the group of volume preserving diffeomorphisms of some fixed probability density $\dx$. 
Thus, constructions (metrics, connections, geodesics) on $\Prob$ can be pulled back to $\Diff$ via the map $\varphi\mapsto \varphi_*\dx$.\footnote{Note that, for the Wasserstein geometry on the space $\Prob$, the connection to Riemannian geometry on the diffeomorphism group was worked out in detail in the celebrated work of Otto~\cite{otto2001geometry}, whereas for the Fisher--Rao geometry introduced in the next section, the analogous construction, together with its differences from the Wasserstein setting, was described by Khesin et al.~\cite{khesin2013geometry}.}

For $a\in \Omega^n(M)$ and $\mu\in \Dens$, we denote by $\frac{a}{\mu}$ the Radon--Nikodym derivative of $a$ with respect to $\mu$. Given a fixed (probability) density $\lambda\in \Prob$ the map $\mu \mapsto \frac{\mu}{\dx}$ allows us to identify $\Dens$ with positive smooth functions on $M$, and $\Prob$ with the positive smooth functions that integrate to one with respect to $\lambda$.
Under this identification, the tangent space to $\Dens$ consists of all smooth functions on $M$, and the tangent space to $\Prob$ consists of all smooth functions whose integral (with respect to $\lambda$) is zero.

\begin{remark}
We note that $\Dens$ and $\Prob$ could be defined on a much rougher base space than a smooth manifold $M$ (as we assume here), and so are the $\alpha$-connections and the $\alpha$-Fisher--Rao metrics defined below. 
Many of the results in this work naturally generalize to rougher base spaces.
Yet for discussing invariance under the action of \eqref{eq:pushforward}, as well as for the hydrodynamic interpretation and the relation to the Proudman--Johnson equations, the differential structure of the base manifold is required.
\end{remark}

\subsection{The Fisher--Rao metric}

Next we introduce the central object of information geometry: the Fisher--Rao metric, a Riemannian metric defined on the space of (probability) densities. It is an infinite-dimensional version of the historical Fisher--Rao metric, introduced by Rao \cite{rao1945} on parametric families of probability distributions using the Fisher information, in the sense that it induces Rao's metric on the corresponding submanifolds.
\begin{definition}[Fisher--Rao metric]
    Given $\mu\in \Dens$ and $a,b\in T_\mu\Dens$ the Fisher--Rao metric on $\Dens$ is given by
    \begin{align}\label{Fisher--Rao}
    \GFR_{\mu}(a,b)=\int \frac{a}{\mu} \frac{b}{\mu}\mu\;.
    \end{align}
    Via restriction $\GFR$ induces a Riemannian metric on $\Prob$, which we denote by the same letter.
\end{definition}
The Fisher--Rao metric is invariant under the action of the diffeomorphism group, i.e.,
$$G_{\varphi_*\mu}(\varphi_*a,\varphi_*b)=G_{\mu}(a,b)$$ for all (probability) densities $\mu$, tangent vectors $a,b$ and diffeomorphisms $\varphi$. 
\v{C}encov's theorem states that, on the space of probability densities, it is indeed the unique Riemannian metric with this property\footnote{In the setting of the present article this requires that $\operatorname{dim}(M)>1$, see~\cite{bauer2016uniqueness}; assuming invariance under the larger set of all sufficient statistics, as used, e.g., in~\cite{ay2015information}, this condition is not needed.}, cf.~\cite{cencov1982,ay2015information,bauer2016uniqueness}.

\subsection{The Amari--\v{C}encov $\alpha$-connections on $\Dens$}
 
Now we introduce the $\alpha$-connections on the space $\Dens$, a family of affine connections that are dual with respect to the Fisher--Rao metric, and such that the $0$-connection is the Levi-Civita connection. We will follow the presentation in~\cite{bauer2024p}, although here we extend the definition to any $\alpha\in\R$. In the finite-dimensional case, i.e., when $M$ is a finite set, the definitions below coincide with the classical ones, see, e.g., \cite{amari2000methods,ay2015information}.
\begin{definition}[$\alpha$-connections on $\Dens$]\label{lem:alphacon}
    For any $\alpha\in \mathbb R$, the $\alpha$-connection $\nabla^{(\alpha)}$ on $\Dens$ is given by
    \begin{equation}\label{alpha-dens}
    \nabla^{(\alpha)}_a b = Db.a - \frac{1+\alpha}{2} \frac{a}{\mu}b,\quad a,b\in T_\mu\Dens.
    \end{equation}
    Here $Db.a|_{\mu}:=D_\mu b(a_\mu)$ denotes the directional derivative of the vector field $b$ in the direction given by $a_\mu$. 
\end{definition}
The easiest way to read this definition (and similar formulae below) is to consider again the identification of densities and positive functions via $\mu\mapsto \mu/\dx$. 
The $\alpha=-1$ and $\alpha=1$ connections are called the mixture and exponential connections, respectively. In the following lemma we will collect several key properties for the $\alpha$-connections on the space  $\Dens$.
\begin{lemma}[Properties of $\alpha$-connections on $\Dens$]\label{lem:alphacon_dens}
    Let $\nabla^{(\alpha)}$ be as defined in~\eqref{alpha-dens}. We have:
    \begin{enumerate}[label=$(\alph*)$]
    \item\label{dens_prop1}For any $\alpha\in \mathbb R$, $\nabla^{(\alpha)}$ is invariant under the action \eqref{eq:pushforward} of $\Diff$ on $\Dens$. 
    \item For any $\alpha\in \R$, $\nabla^{(\alpha)}$ is torsionless.
    \item\label{dens_prop2} For $\alpha\in \mathbb R$, $\nabla^{(\alpha)}$ and $\nabla^{(-\alpha)}$ are dual connections on $(\Dens,\GFR)$, i.e.,
    \[
    D(G(a,b)).c = G(\nabla^{(\alpha)}_c a,b) + G(a,\nabla^{(-\alpha)}_c b).
    \]
    For $\alpha=0$ we obtain that $\nabla^{(0)}$ is self-dual (i.e., metric) and thus coincides with the Levi-Civita connection of the Fisher--Rao metric.  
    \item\label{dens_prop3} For $\alpha\in (-1,1)$, $\nabla^{(\alpha)}$ 
    can be interpreted as the connection associated to the $\alpha$-divergence $D^{(\alpha)}$, i.e.,
    $$G(\nabla_a^{(\alpha)}b,c)=-\left.\partial_\mu(\partial_\mu\partial_\nu D^{(\alpha)}(\mu||\nu)[b,c])[a]\right|_{\nu=\mu},$$
    where $D^{(\alpha)}:\Dens\times\Dens\to\mathbb R$ is defined via
    \begin{equation}\label{eq:alphadiv}
    D^{(\alpha)}(\mu||\nu) = \frac{2}{1-\alpha}\int_M \nu + \frac{2}{1+\alpha}\int_M \mu -\frac{4}{(1-\alpha)(1+\alpha)}\int_M \left(\frac{\mu}{\dx}\right)^\frac{1-\alpha}{2}\left(\frac{\nu}{\dx}\right)^{\frac{1+\alpha}{2}}\dx.
    \end{equation}
    \item\label{dens_prop4} For $\alpha\in (-1,1)$,  $\nabla^{(\alpha)}$ can be identified with the Chern connection of an invariant Finsler metric on $\Dens$, called the $L^p$-Fisher--Rao metric, where $p=\frac{2}{1-\alpha}$.
    \item\label{dens_prop5} 
    For $\alpha\in (-1,1)$ and any initial conditions $\mu_0\in \Dens$ and $a\in T_\mu \Dens$, the unique $\alpha$-connection geodesic $\mu:[0,T)\to \Dens$
    defined on its maximal interval of existence $[0,T)$ is given by
    \begin{equation*}
    \mu(t)=\left(\left(\frac{\mu_0}{\dx}\right)^{\frac{1-\alpha}{2}}+
    t\left(\frac{\mu}{\dx}\right)^{-\frac{1+\alpha}{2}}\frac{a}{\dx}\right)^{\frac{2}{1-\alpha}} \dx.
    \end{equation*}
    The geodesic $\mu(t)$ exists for all time $t$, i.e., $T=\infty$, if and only if $\frac{a}{\dx}(x)\geq0$ for all $x\in M$.
    Thus $\nabla^{(\alpha)}$ is  geodesically incomplete since there exist geodesics that leave $\Dens$ in finite time. 
    \end{enumerate}
\end{lemma}
\begin{proof}
Properties~\ref{dens_prop1}--\ref{dens_prop3} follow by direct calculation, whereas properties~\ref{dens_prop4}--\ref{dens_prop5} have been shown in our previous article~\cite{bauer2024p}. For similar results in the finite dimensional setting we refer to~\cite[Section 2.5.2.]{ay2017information}.
\end{proof}

\subsection{The Amari--\v{C}encov $\alpha$-connections on $\Prob$}

Next we introduce the $\alpha$-connections on the space of probability densities, which we will denote by  $\nablabar^{(\alpha)}$. 
These are obtained from the $\alpha$-connections on $\Dens$, by projecting them on $\Prob$ with respect to the Fisher--Rao metric. Explicitly, we have: 
\begin{definition}[$\alpha$-connections on $\Prob$]
    For any $\alpha\in \mathbb R$, the $\alpha$-connection $\nablabar^{(\alpha)}$ on $\Prob$ is given by
    \begin{equation}\label{alpha-prob}
    \nablabar^{(\alpha)}_a b = Db.a - \frac{1+\alpha}{2} \left(\frac{a}{\mu}b -\left(\int_M \frac{a}{\mu}\frac{b}{\mu} \mu\right)\mu\right).
    \end{equation}
\end{definition}
For finite sample spaces this formula for the $\alpha$-connection is well-known (e.g., \cite[Section~2.5.2]{ay2017information}); in infinite dimensions formula \eqref{alpha-prob} agrees with the formula (22) in \cite{lenells2014amari}, under the identification of $\Prob = \Diff/\SDiff$. In the following lemma we will collect several key properties for the $\alpha$-connections on $\Prob$:
\begin{lemma}[Properties of $\alpha$-connections on $\Prob$]\label{lem:alphacon_prob}
    Let $\nablabar^{(\alpha)}$ be as defined in~\eqref{alpha-prob}. We have:
    \begin{enumerate}[label=$(\alph*)$]
    \item\label{prob_prop1} For any $\alpha\in \mathbb R$, $\nablabar^{(\alpha)}$ is invariant under the action \eqref{eq:pushforward} of $\Diff$ on $\Prob$. 
    \item For any $\alpha\in \R$, $\nablabar^{(\alpha)}$ is torsionless.
    \item\label{prob_prop2} For $\alpha\in \mathbb R$, $\nablabar^{(\alpha)}$ and $\nablabar^{(-\alpha)}$ are dual connections on $(\Prob,\GFR)$. 
    For $\alpha=0$ we obtain that $\nablabar^{(0)}$ is self-dual and thus coincides with the Levi-Civita connection of the Fisher--Rao metric.  
    \item\label{prob_prop3} For $\alpha\in (-1,1)$, $\nablabar^{(\alpha)}$ 
    can be interpreted as the connection associated to the restriction to $\Prob$ of the $\alpha$-divergence $D^{(\alpha)}$, as defined in~\eqref{eq:alphadiv}.
    \item\label{prob_prop4} For $\alpha\in \mathbb R$, the connection  $\nablabar^{(\alpha)}$ on $\Prob$ is the orthogonal projection of the connection $\nabla^{(\alpha)}$ on $\Dens$ with respect to the Fisher--Rao metric $\GFR$. 
    \end{enumerate}
\end{lemma}
\begin{proof}
All of these properties follow by direct calculation. For similar results in the finite dimensional setting we refer to~\cite[Section 2.5]{ay2017information}. 
\end{proof}

\begin{remark}
Note that the Finsler structure associated to the $\alpha$-connections no longer holds on $\Prob$, expect for the trivial case $\alpha=0$.
\end{remark}

\section{The $\alpha$-Fisher--Rao metric and the $\alpha$-connections on $\Dens$}\label{sec:dens}

In this section we will present the first main result of this article: namely, we introduce a new family of Riemannian metrics on the space $\Dens$, which we will construct in such a way that the induced Levi-Civita connections coincide with the $\alpha$-connections. As a by-product of our construction we will extend in Theorem~\ref{thm:explicit_dens} the explicit formula for geodesics from Lemma~\ref{lem:alphacon_dens} to arbitrary~$\alpha\in\mathbb R$. Finally, in Section~\ref{sec:PJnonperiodic} we will extend our construction to the space of densities on the non-compact manifold $M=\mathbb R$, which will allow us to find an interpretation of the non-periodic, generalized Proudman--Johnson equation as an Euler--Arnold-type equation.

\subsection{A Riemannian interpretation for the $\alpha$-connections}

We  start by defining the $\alpha$-Fisher--Rao metric:
\begin{definition}[$\alpha$-Fisher--Rao metric]
    Let $\alpha\in \mathbb R$ and let $\dx$ be a fixed background density on $\Dens$. Given $\mu\in \Dens$ and $a,b\in T_\mu\Dens$ we define the $\alpha$-Fisher--Rao metric via:
    \begin{align}\label{p-Fisher--Rao}
    G^{\alpha}_{\mu}(a,b)=\int \left(\frac{\mu}{\dx}\right)^{-\alpha-1} \frac{a}{\dx}\frac{b}{\dx}\dx\;.
    \end{align}
    Note that $\alpha=0$ corresponds to the standard Fisher--Rao metric as defined in~\eqref{Fisher--Rao}, i.e., $G^0=G$.
\end{definition}

\begin{remark}[Invariance of the $\alpha$-Fisher--Rao metric]
Note that $\frac{\mu}{\dx}$ is a strictly positive function and thus the above indeed defines a smooth Riemannian metric, i.e., it is positive definite.
However, the metric is \emph{not} invariant under the action \eqref{eq:pushforward} of $\Diff$ 
unless $\alpha=0$, for which it is equal to the Fisher--Rao metric. 
To see the non-invariance for $\alpha\neq 0$ we just note the dependance on the background density $\lambda$, which is an obstruction to invariance with respect to the full diffeomorphism group. 
These metrics are, however, invariant under the action of $\operatorname{Diff}_{\lambda}(M)$, the group of all $\lambda$-preserving diffeomorphisms. 
This can be easily seen by applying a $\lambda$-preserving diffeomorphism $\varphi$, using that $\lambda = \varphi^*\lambda$ and then applying a change of variables by $\varphi$ in the integral. 
Note that this is analogous to the situation of the $L^2$-Wasserstein metric in optimal transport, which corresponds to a $\operatorname{Diff}_{\lambda}(M)$-invariant metric on $\Diff$, that is not invariant under the action of the full diffeomorphism group. 
\end{remark}

Next, we show that the $\alpha$-Fisher--Rao metric satisfies the desired property, i.e., we show that the covariant derivative associated to the $\alpha$-Fisher--Rao metric on $\Dens$ is equal to the Amari--\v{C}encov $\alpha$-connection.
\begin{theorem}[Levi-Civita connection of the $\alpha$-Fisher--Rao metric on $\Dens$]\label{thm:alphaFR_connection_dens}
    For any $\alpha\in \mathbb R$, the Levi-Civita connection of the $\alpha$-Fisher--Rao metric on the space of densities $\Dens$ is given by
    \begin{equation*}
    \nabla_a b = Db.a  -\frac{1+\alpha}{2} \frac{a}{\mu}\frac{b}{\mu} \mu,
    \end{equation*}
    which coincides with the  $\alpha$-connection as introduced in Definition~\ref{lem:alphacon}.
\end{theorem}
\begin{proof}
The proof follows by direct calculation: we have $\nabla_ab=Db.a+\Gamma(a,b)$, where the Christoffel map $\Gamma(a,b)$ can be computed using
\begin{equation*}
G^{\alpha}_{\mu}(c,\Gamma_{\mu}(a,b))= \frac12\left(D_{\mu,a} G^{\alpha}\right)(b,c)+\frac12\left(D_{\mu,b} G^{\alpha}\right)(a,c)-\frac12\left(D_{\mu,c} G^{\alpha}\right)(a,b),
\end{equation*}
where $D_{\mu,a} G^{\alpha}$ denotes the variation of the $\alpha$-Fisher--Rao metric in direction $a$. We calculate
\begin{align*}
\left(D_{\mu,a} G^{\alpha}\right)(b,c)=  
D_{\mu,a}\left(\int \left(\frac{\mu}{\dx}\right)^{-\alpha-1} \frac{b}{\dx}\frac{c}{\dx}\dx\right)
=-\int (\alpha+1)\left(\frac{\mu}{\dx}\right)^{-\alpha-2}\frac{a}{\dx}\frac{b}{\dx}\frac{c}{\dx}\dx\;.
\end{align*}
Thus we get 
\begin{align*}
&\frac12 \left(D_{\mu,a} G^{\alpha}\right)(b,c)+\frac12\left(D_{\mu,b} G^{\alpha}\right)(a,c)-\frac12\left(D_{\mu,c} G^{\alpha}\right)(a,b)\\&\qquad=
-\frac12\int (\alpha+1)\left(\frac{\mu}{\dx}\right)^{-\alpha-2}\frac{a}{\dx}\frac{b}{\dx}\frac{c}{\dx}\dx=
-\frac12\int (\alpha+1)\left(\frac{\mu}{\dx}\right)^{-\alpha-1}\frac{a}{\mu}\frac{b}{\dx}\frac{c}{\dx}\dx\\
&\qquad=G\left(-\frac{\alpha+1}{2} \frac{a}{\mu}b,c\right).
\end{align*}
Consequently we obtain the desired formula for the Christoffel symbols. 
\end{proof}
This interpretation of the $\alpha$-connections as Levi-Civita connections of a Riemannian metric leads to a new interpretation of the geodesics of the $\nabla^{(\alpha)}$ as energy-minimizing curves, implying in particular that the corresponding Riemannian energy is conserved along geodesics:
\begin{corollary}[Conservation of the Riemannian Energy]
Let $\alpha\in \mathbb R$ and let $\mu:[0,1]\to \Dens$ be a geodesic of the $G^{\alpha}$-metric (of the $\alpha$-connection, equivalently). Then the (instantaneous) Riemannian $G^{\alpha}$-energy
\begin{align*}
E(\mu)=G^{\alpha}_{\mu}(\mu_t,\mu_t)=
\int \left(\frac{\mu}{\dx}\right)^{-\alpha-1} \frac{\mu_t}{\dx}\frac{\mu_t}{\dx}\dx\
\end{align*}
is constant in time. 
\end{corollary}
Next we show that the $\alpha$-Fisher--Rao metric can be obtained as the pullback of the flat $L^2$-metric on the space of smooth functions, thereby leading to explicit formulas for geodesics on the space of densities:
\begin{theorem}\label{thm:explicit_dens}
    For fixed $\dx$ we consider the mapping $\Phi_{\alpha}: 
    \Dens \to C^{\infty}(M)$, defined via 
    \begin{equation}\label{eq:Phi}
    \Phi_{\alpha}: \begin{cases}
    \mu &\mapsto \log\left(\frac{\mu}{\dx}\right), \qquad\qquad\,\alpha = 1\\
    \mu &\mapsto \frac{2}{|1-\alpha|}\left(\frac{\mu}{\dx}\right)^{\tfrac{1-\alpha}2}, \qquad \alpha \neq 1.
    \end{cases}
    \end{equation}
    Then the pullback via $\Phi_{\alpha}$ of the $L^2(\dx)$ Riemannian metric on $C^{\infty}(M)$ is the $\alpha$-Fisher--Rao metric. Furthermore, for any $\alpha\neq 1$ the image $\Phi_{\alpha}(\Dens)$ is the space of all smooth, positive functions, whereas for $\alpha=1$ it coincides with all of $C^{\infty}(M)$.
    In particular, $G^\alpha$ is flat for all $\alpha\in \mathbb{R}$.
\end{theorem}
Note that the maps $\Phi_\alpha$ as defined here differ from their analogues in \cite{bauer2024p} by the factor $\frac{2}{|1-\alpha|}$.

\begin{proof}
To calculate the pullback metric we only need to calculate the variation of $\Phi_{\alpha}$ at $\mu\in\Dens$ in direction $a\in T_{\mu}\Dens$. We have, for $\alpha\ne 1$:
\begin{equation*}
T_\mu\Phi_{\alpha}(a)=\operatorname{sgn}(1-\alpha)\left(\frac{\mu}{\dx}\right)^{-\tfrac{1+\alpha}2}\frac{a}{\dx},
\end{equation*}
and thus the desired formula follows by observing that 
\begin{equation*}
G^{\alpha}_\mu(a,b)=\int T_\mu\Phi_{\alpha}(a).T_\mu\Phi_{\alpha}(b) \dx.
\end{equation*}
The calculation for $\alpha=1$ is similar.
\end{proof}
Using that geodesics in $C^{\infty}(M)$ w.r.t.\ the $L^2(\dx)$-metric are straight lines, see e.g.~\cite{bruveris20182}, a direct consequence of  the above results are the following explicit formulas for geodesics of all $\alpha$-connections on $\Dens$; for $\alpha\in(-1,1)$ this formula was already shown in~\cite{bauer2024p}.
\begin{corollary}
Let $\alpha \in \mathbb R$. Given any initial conditions $\mu_0\in \Dens$ and $a\in T_\mu \Dens$, the unique $\alpha$-connection geodesic $\mu:[0,T)\to \Dens$
    defined on its maximal interval of existence $[0,T)$ is given by
    \begin{equation*}
    \mu(t)=\begin{cases}
    \operatorname{exp}\left(
    t\left(\frac{\mu_0}{\dx}\right)^{-1}
    \frac{a}{\dx}\right) \dx,\qquad \alpha=1\\
      \left(\left(\frac{\mu_0}{\dx}\right)^{\frac{1-\alpha}{2}}+
    t\left(\frac{\mu_0}{\dx}\right)^{-\frac{1+\alpha}{2}}\frac{a}{\dx}\right)^{\frac{2}{1-\alpha}} \dx, \qquad \alpha\neq 1.   
    \end{cases}
    \end{equation*}
    For $\alpha\neq 1$ the geodesic $\mu(t)$ exists for all time $t$, i.e., $T=\infty$, if and only if $\frac{a}{\dx}(x)\geq0$ for all $x\in M$. 
    Thus for $\alpha\neq 1$, $\nabla^{(\alpha)}$ is  geodesically incomplete since there exist geodesics that leave $\Dens$ in finite time.  For $\alpha=1$ all geodesics exist for all time $t$ and thus $\nabla^{(1)}$  is geodesically complete.    
\end{corollary}

\subsection{The Proudman--Johnson equations and the non-invariant $\dot H^1$-metric on $\DiffR$}\label{sec:PJnonperiodic}

We conclude this section by describing an analogous construction on the group of diffeomorphisms on the real line, which turns out to be related to the family of inviscid, generalized Proudman--Johnson equations with parameter $\alpha$ (henceforth gPJ equation), given by
\begin{equation*}\label{eq:pj}
u_{txx} + (2-\alpha)u_xu_{xx} +uu_{xxx} = 0, 
\end{equation*}
where $u$ is a function on the real line. The interest in these equations can be found in the area of mathematical hydrodynamics, as they contain several prominent PDEs as special cases: the original Proudman--Johnson equation, in which $\alpha = 3$, corresponds to axisymmetric Navier--Stokes equations in $\R^2$; the generalized equations, which were first proposed in \cite{okamoto2000some}, contain several other important special cases, in particular the Hunter--Saxton equation for $\alpha=0$, the $\mu$-Burgers equation $\alpha=-1$, and self-similar axisymmetric Navier--Stokes equations in higher dimensions.
See \cite{Wun11} for further information about the equation and its motivation.

The relation of these equations to the $\alpha$-connections studied here has been first described by Lenells and Misio{\l}ek in the periodic setting~\cite{lenells2014amari}: namely they considered the mapping $\varphi\mapsto \varphi_*d\theta$, which maps the diffeomorphism group of the circle to the space of probability densities on $S^1$. They showed that the pullback of the geodesic equations of the $\alpha$-connections via this mapping yields the periodic gPJ equations with parameter $\alpha$. 
The resulting connection reduces to a Levi-Civita connection of the homogenous $W^{1,2}$-Riemannian metric for $\alpha=0$, thus recasting the geometric interpretation of the Hunter--Saxton equation as an Euler--Arnold equation, as described in \cite{lenells2007hunter}. In the non-periodic setting we have shown in~\cite{bauer2022geometric}, that one can interpret the gPJ equations as Euler--Arnold equations of a Finsler-metric if $\alpha\in (-1,1)$. 

In this section, we will consider the pullback of the $\alpha$-Fisher--Rao metric to show that one can indeed interpret the whole family of gPJ equations as Euler--Arnold-type equations, albeit for a non-invariant Riemannian metric. 
We note that the results in this section do not follow directly from the above calculations on $\Dens$ as we previously assumed that the manifold $M$ is compact, whereas here we treat the case $M=\mathbb R$. The principal distinction in the non-compact setting considered in this section is that suitable decay conditions must be imposed on the associated function spaces. Observe that without these assumptions, several key quantities would not be well defined; in particular, the $\alpha$-Fisher--Rao metric would not be finite on the whole tangent bundle. See \cite[Sec.~2.1]{bauer2022geometric} for a discussion on the appropriate decay conditions.

We start by introducing an appropriate group of diffeomorphisms on the real line, where we have to pay specific attention to the decay conditions towards infinity: we let $\DiffRshift$ denote the group of diffeomorphisms that converge to the identity as $x\to -\infty$. More precisely, we define
\begin{equation*}
    \DiffRshift:=\left\{\varphi=x+f(x): f'\in W^{\infty,1}(\mathbb R),\; f'>-1, \text{ and }\lim_{x\to-\infty}f(x)=0\right\}\;,
\end{equation*}
where $W^{\infty,1}(\mathbb R)=\cap_{k\in\mathbb{N}} W^{k,1}(\mathbb R)$ is the intersection of all Sobolev spaces. 
The space $\DiffRshift$ has been first considered in~\cite{bauer2014homogeneous} as a modeling space for the Hunter-Saxton equation on the real line and it has been shown to be an infinite-dimensional Lie group with Lie algebra
\begin{equation*}
    \DiffRLie:=\left\{u: u'\in W^{\infty,1}(\mathbb R)\text{ and }\lim_{x\to-\infty}u(x)=0\right\}\;.
\end{equation*}
We are now ready to define the $\alpha\dot H^1$ Riemannian metric, which will yield the desired geometric description for the gPJ equations:

\begin{definition}
    Let $\varphi\in \DiffRshift$ and $u\circ\varphi,v\circ\varphi \in T_\varphi\DiffRshift$ with $u,v\in C_1^{\infty}(\R)$.  The $\alpha\dot H^1$-metric is given by
    \begin{align*}\label{alpha-doth1}
    G^{\alpha\dot H^1}_{\varphi}(u\circ\varphi,v\circ\varphi)=\int \varphi_x^{-\alpha-1}(u\circ\varphi)_x.(v\circ\varphi)_x\;\dx
    =\int \varphi_x^{-\alpha+1}(u_x\circ\varphi).(v_x\circ\varphi)\;\dx,
    \end{align*}
    where $\dx$ denotes the Lebesgue measure on $\mathbb R$. 
    Note that the decay conditions on $\varphi, u$ and $v$ guarantee the well-posedness of the above definition. 
\end{definition}
Using the change of coordinate formula it follows that the metric $G^{\alpha\dot H^1}$ is right-invariant if and only if $\alpha=0$.

Next we will show that the $\alpha\dot H^1$-metric can be obtained as a pullback of the $\alpha$-Fisher--Rao metric. 
To this end, we first need to introduce an appropriate space of densities on the real line, which we will define as the image of $\DiffRshift$ under the map $\varphi \mapsto \varphi_*d\theta$. We obtain the space:
\begin{equation*}
\DensR:=\left\{h\dx: (h-1)\in W^{\infty,1}(\mathbb R)\text{ and } h>0\right\}\;.
\end{equation*}
It follows that $\DensR$ is an infinite-dimensional manifold with tangent space
\begin{equation*}
T_{\mu}\DensR:=\left\{r\dx: r\in W^{\infty,1}(\mathbb R)\right\}\;.
\end{equation*}
Using these definitions, it follows that the $\alpha$-Fisher--Rao metric, as defined in~\eqref{p-Fisher--Rao} for densities on compact manifolds, can be extended to $\DensR$, and by direct calculation we obtain the following result:
\begin{lemma}\label{lem:pullback}
    Consider the map 
    \begin{equation*}
    \Theta: \begin{cases}
    \DiffRshift&\to\DensR\\
    \varphi&\mapsto \varphi_*\dx.
    \end{cases}
    \end{equation*}
    The pullback of the $\alpha$-Fisher--Rao metric via $\Theta$ is equal to the $\alpha\dot H^1$-metric on $\DiffRshift$, i.e., for all $u\circ\varphi,v\circ\varphi\in T_{\varphi}\DiffRshift$ we have
    \begin{equation*}
    G^{\alpha\dot H^1}_{\varphi}(u\circ\varphi,v\circ\varphi)=(\Theta^*G^{\alpha})_{\varphi}(u\circ\varphi,v\circ\varphi)\;.
    \end{equation*}
\end{lemma}
\begin{proof}
This follows by direct calculation.
\end{proof}
We are now prepared to present the main result of this subsection, which, given the above results, should come as no surprise.
\begin{theorem}
    Let $\alpha\in \mathbb R$. The generalized Proudman--Johnson equations with parameter $\alpha$ are  Euler--Arnold-type equations of the (non-invariant) Riemannian metric $G^{\alpha\dot H^1}$ on $\DiffRshift$. Consequently, the generalized-Proudman--Johnson equations are globally well-posed on $\DiffRLie$ if and only if $\alpha=1$. 
\end{theorem}
\begin{proof}
The proof of this result follows from Lemma~\ref{lem:pullback} and the fact that the geodesic equations of the $\alpha$-connections on $\DensR$ are equivalent to the gPJ equation. 
\end{proof}

\section{The $\alpha$-Fisher--Rao metric and the $\alpha$-connections on $\Prob$}\label{sec:prob}

\subsection{Metricity of the $\alpha$-connections on $\Prob$}

Here we consider the restriction of the $\alpha$-Fisher--Rao metric to a Riemannian metric on $\Prob$. We will, however, show that, in general, the Levi-Civita connection of this Riemannian metric does not coincide with the Amari-\u{C}encov $\alpha$-connection on this space. The difference stems from the fact that the $\alpha$-connection on $\Prob$ is the orthogonal projection of the $\alpha$-connection on $\Dens$ with respect to the Fisher--Rao metric~\cite{bauer2024p}, while the induced connection of the restriction of the $\alpha$-Fisher--Rao metric to $\Prob$ is the orthogonal projection of the $\alpha$-connection on $\Dens$ with respect to the $\alpha$-Fisher--Rao metric. 
These two only coincide for $\alpha=0$ and $\alpha=-1$. Indeed, the main result of this section shows that the Amari-\u{C}encov $\alpha$-connection on $\Prob$ is the Levi-Civita connection of a Riemannian metric if and only if $\alpha\in\{-1,0,1\}$: for $\alpha=0$ and $\alpha=-1$ this Riemannian metric is the restriction of the $\alpha$-Fisher--Rao metric, and for $\alpha=1$ it is a slight variant of the restricition of the $1$-Fisher--Rao metric. We start by calculating the Levi-Civita 
connection of the $\alpha$-Fisher--Rao metric:
\begin{lemma}[Levi-Civita connection of the $\alpha$-Fisher--Rao metric on $\Prob$]\label{thm:geo_alphaFR_prob}
    The Levi-Civita connection of the $\alpha$-Fisher--Rao metric on the space of densities $\Prob$ is given by
    \begin{equation*}
    \tilde \nabla_a b = Db.a  -\frac{\alpha+1}{2} \left( \frac{a}{\mu}b-\frac{\int \frac{a}{\mu}\frac{b}{\mu} \mu}{\int \left(\frac{\mu}{\dx}\right)^{\alpha+1}\dx}\left(\frac{\mu}{\dx}\right)^{\alpha+1}\dx\right),
    \end{equation*}
    which coincides with the  $\alpha$-connection as defined in Definition~\ref{lem:alphacon} if and only if $\alpha\in\{0,-1\}$.
\end{lemma}

\begin{proof}
To calculate the induced connection on the submanifold $\Prob$, we only need to calculate the orthogonal projection of the $\alpha$-Fisher--Rao metric. 
For every $\mu\in \Prob$ the tangent space $T_{\mu}\Prob$ at $\mu$  is given by all $\delta \mu$ such that $\int \delta \mu=0$ (independently of $\alpha$). 
Let $a\in T_{\mu}\Dens$. We claim that  any element in the orthogonal complement can be written as $a=C\left(\frac{\mu}{\dx}\right)^{\alpha+1}\dx$ with $C\in \mathbb R$ being a constant. It is easy to see that such $a$ is indeed orthogonal to all $\delta \mu \in T_{\mu}\Prob$ as 
\begin{equation*}
G^{\alpha}_{\mu}(a,\delta \mu)=\int \left(\frac{\mu}{\dx}\right)^{-\alpha-1}  
C\left(\frac{\mu}{\dx}\right)^{\alpha+1} \frac{\dx}{\dx}\frac{\delta \mu}{\dx}\dx=\int \delta \mu=0\;.
\end{equation*}
It remains to show that this really spans all of the complement, but this is clear since for any $a\in T_{\mu}\Dens$
\begin{equation*}
a-C\left(\frac{\mu}{\dx}\right)^{\alpha+1}\dx \in T_{\mu}\Prob,
\,\text{ where }\,\, C=\frac{\int a}{\int \left(\frac{\mu}{\dx}\right)^{\alpha+1}\dx}.
\end{equation*}
Note that if $\alpha=0$, then $C=\int a$ as the denominator integrates to one.
\end{proof}

Next we formulate the main result of this section, which shows that the $\alpha$-connections on $\Prob$ cannot be viewed as Levi-Civita connections of a Riemannian metric if $\alpha \notin \{-1,0,1\}$. Note that the Riemannian interpretation for $\alpha \in \{-1,0\}$ follows directly from the above theorem, but that the case $\alpha=1$ needs a slight adaption of the construction of the corresponding Riemannian metric. 
\begin{theorem}\label{thm:alpha_nonmetric}
    The $\alpha$-connection $\nabar$ on $\Prob$ is the Levi-Civita connection of a Riemannian metric if and only if $\alpha\in\{-1,0,1\}$. More precisely we have:
    \begin{itemize}
    \item If $\alpha\in \{-1,0\}$ then $\nabar$ is the Levi-Civita connection 
    of the restriction of the $\alpha$-Fisher--Rao metric to $\Prob$ .
    \item If $\alpha=1$ then $\nabar$ is the Levi-Civita connection of the Riemannian metric 
    \begin{equation}\label{eq:tildeG1}
    \tilde G^1_{\mu}(a,b)=\int_M\frac{a}{\mu}\frac{b}{\mu}\dx -\int_M\frac{a}{\mu}\dx \int \frac{b}{\mu}\dx.
    \end{equation}
    \item If $\alpha \notin \{-1,0,1\}$ then $\nabar$ is a non-metric affine connection.
    \end{itemize}
\end{theorem}
\begin{remark}[Interpreting the Riemannian metric for $\alpha=1$ as a pullback metric]
\label{rem:alpha1} As mentioned above, the metric for $\alpha=1$ does not coincide with the (restriction) of the $1$-FR metric as introduced in the previous section. Instead of the pullback via the map $\Phi_1$, one has to slightly adapt the mapping and define $\tilde\Phi_1:\Prob\to C^{\infty}(M)$ via
\begin{equation*}
\tilde\Phi_1(\mu)= \log\left(\frac{\mu}{\lambda}\right)-\int \log\left(\frac{\mu}{\lambda}\right)\lambda.
\end{equation*}
Using this definition, the pullback of the flat $L^2$-metric on $C^{\infty}(M)$ is exactly the metric $\tilde{G}^1$ as defined above. Note that the image of $\tilde{\Phi}_1$ is the linear space of all zero average functions, whose geodesics with respect to the $L^2$-metric are simply straight lines and exist for all time; thus, the geodesics of $\tilde{G}^1$ (i.e., of $\overline{\nabla}^{(1)}$) are the pullback of these straight lines, and the space is geodesically complete.
This map and the characterization of the geodesics (though not the Riemannian metric \eqref{eq:tildeG1}) were studied in \cite{lenells2014amari}. We give the expression for the $\overline\nabla^{(1)}$-geodesics in Proposition~\ref{prop:alpha1}.
\end{remark}

To prove Theorem~\ref{thm:alpha_nonmetric} we first need some lemmas.
\begin{lemma}\label{lem:deriv}
    For any $a,b,c$ smooth vector fields on $\Prob$ and $\mu\in\Prob$,
    $$D_\mu\big(G(b,c)\big).a_\mu=G_\mu(D_\mu b.a_\mu,c_\mu)+G_\mu(b_\mu,D_\mu c.a_\mu)-\int_M \frac{a}{\mu}\frac{b}{\mu}\frac{c}{\mu}\mu.$$
\end{lemma}
\begin{proof}
Consider a smooth curve $\mu(t)$ in $\Prob$ such that $\mu(0)=\mu$ and $\dot\mu(0)=a_\mu$, and identify $\mu,a_\mu,b_\mu,c_\mu$ with their Radon-Nykodym derivatives with respect to $\lambda$. Then 
\begin{align*}
D_\mu\big(G(b,c)\big).a_\mu=\left.\frac{d}{dt}\right|_{t=0}G_{\mu(t)}(b_{\mu(t)},c_{\mu(t)})=\int_M\left.\frac{d}{dt}\right|_{t=0} \left(\frac{b_{\mu(t)}}{\mu(t)}\frac{c_{\mu(t)}}{\mu(t)}\mu(t)\right)\lambda,
\end{align*}
since $M$ is compact. Computing the integrand using the chain rule,
$$\left.\frac{d}{dt}\right|_{t=0} \left(\frac{b_{\mu(t)}}{\mu(t)}\frac{c_{\mu(t)}}{\mu(t)}\mu(t)\right)=\left.\frac{d}{dt}\right|_{t=0} b_{\mu(t)}c_{\mu(t)}\mu(t)^{-1}=(D_\mu b.a_\mu)c_\mu\mu^{-1}+b_\mu(D_\mu c.a_\mu)\mu^{-1}-b_\mu c_\mu a_\mu \mu^{-2},$$
which yields the result as claimed.
\end{proof}
We now compute the Riemann curvature tensor of $\nabar$.
\begin{lemma}\label{prop:alpha-curvature}
    For $a,b,c\in T_{\mu}\Prob$  the Riemann curvature tensor of $\nabar$ on $\Prob$ is given by:
    $$R^{(\alpha)}(a,b)c=\frac{1-\alpha^2}{4}\left(\GFR(b,c)a-\GFR(a,c)b\right),$$
    where $\GFR$ is the Fisher--Rao metric.
\end{lemma}
\begin{remark}
This implies the well-known facts that the curvature of $\nabar$ is flat for $\alpha=\pm 1$, that $R^{(\alpha)}=R^{(-\alpha)}$ and that the sectional curvature is $1/4$ for the Fisher--Rao metric, cf.~\cite[Appendix B]{ay2017information}.
\end{remark}
\begin{proof}
Let $\tau = \frac{1+\alpha}{2}$. The connection $\nabar$ on $\Prob$ can be obtained from $\na$ on $\Dens$ as
$$\nabar_ab=\na_ab+\tau\GFR_{\mu}(a,b)\mu, \quad\text{where}\quad \na_ab=Db.a-\tau\frac{a}{\mu}b,$$
so that
\begin{align*}
\nabar_a\nabar_bc&=\nabar_a\left(\na_bc+\tau\GFR_{\mu}(b,c)\mu\right)\\
&=\na_a\left(\na_bc+\tau\GFR_{\mu}(b,c)\mu\right)+\tau\GFR_\mu\left(a, \na_bc+\tau\GFR_{\mu}(b,c)\mu\right)\mu\\
&=\na_a\na_bc+\tau D\left(\GFR_{\mu}(b,c)\mu\right).a-\tau^2\frac{a}{\mu}\GFR_{\mu}(b,c)\mu\\
&\quad+\left(\tau^2\GFR_{\mu}(a,\mu)\GFR_{\mu}(b,c)+\tau\GFR_{\mu}(a,Dc.b)-\tau^2\GFR_{\mu}(a,\frac{b}{\mu}c)\right)\mu.
\end{align*}
Using the fact that $\GFR_{\mu}(a,\mu)=0$ for any $a$ tangent to $\Prob$, and by Lemma~\ref{lem:deriv}
$$D_\mu\left(\GFR(b,c)\mu\right).a=\GFR_{\mu}(Db.a,c)\mu+\GFR_{\mu}(b,Dc.a)\mu-\GFR_{\mu}\left(a,\frac{b}{\mu}c\right)\mu + \GFR_{\mu}(b,c)a,$$
we obtain
\begin{align*}      
\nabar_a\nabar_bc&=\na_a\na_bc+\tau\left(\GFR_{\mu}(a,Dc.b)+\GFR_{\mu}(b,Dc.a)+\GFR_{\mu}(c,Db.a)\right)\\
&\quad+\tau\left(1-\tau\right)\GFR_{\mu}(b,c)a-\tau\left(1+\tau\right)\GFR_{\mu}\left(a,\frac{b}{\mu}c\right)\mu.
\end{align*}
Using  that
$\nabar_ab=\na_ab+\tau\GFR_{\mu}(a,b)\mu$ and that 
$\na$ has zero curvature we have that:
\begin{align*}
R^{(\alpha)}(a,b)c&=\nabar_a\nabar_bc-\nabar_b\nabar_ac-\nabar_{[a,b]}c\\
&=0+\tau\GFR_{\mu}(c,Db.a-Da.b)\mu+\tau\left(1-\tau\right)(\GFR_{\mu}(b,c)a-\GFR_{\mu}(a,c)b)-\tau\GFR_{\mu}([a,b],c)\mu,
\end{align*}
which yields the result as the first and third terms on the righthand side cancel.
\end{proof}

\begin{proof}[Proof of Theorem~\ref{thm:alpha_nonmetric}]
For $\alpha\in \{-1,0\}$ we have shown in Lemma~\ref{lem:alphacon} that  $\nabar$ is the Levi-Civita connection of the $\alpha$-Fisher--Rao metric. Next we will show the statement for $\alpha=1$. To calculate the Christoffel symbols we proceed as in the proof of Theorem~\ref{thm:alphaFR_connection_dens}. Using similar computations as in Lemma~\ref{lem:deriv}, we have: 
\begin{align*}
D_{\mu}\left(\tilde G^{1}(b,c)\right).a_\mu&=  
D_{\mu}\left(\int_M\frac{b}{\mu}\frac{c}{\mu}\dx -\int_M\frac{b}{\mu}\dx \int \frac{c}{\mu}\dx\right).a_\mu
\\&=-2\int \frac{a}{\mu}\frac{b}{\mu}\frac{c}{\mu}\dx+\int_M\frac{a}{\mu}\frac{b}{\mu}\dx \int \frac{c}{\mu}\dx+\int_M\frac{b}{\mu}\dx \int \frac{a}{\mu}\frac{c}{\mu}\dx
\end{align*}
Thus we get 
\begin{align*}
&\frac12 D_{\mu}\left(\tilde G^{1}(b,c)\right).a_\mu+\frac12 D_{\mu}\Big( G^{\alpha}(a,c)\Big).b_\mu-\frac12 D_{\mu} \Big(G^{\alpha}(a,b)\Big).c_\mu \\&\qquad=
-\int \frac{a}{\mu}\frac{b}{\mu}\frac{c}{\mu}\dx+\int_M\frac{a}{\mu}\frac{b}{\mu}\dx \int \frac{c}{\mu}\dx  \\
&\qquad=\tilde G^{1}_\mu\left(\frac{a}{\mu}\frac{b}{\mu}\mu-\left(\int_M \frac{a}{\mu}\frac{b}{\mu}\mu\right)\mu,c\right).
\end{align*}
Thus we read off that 
\begin{equation*}
\Gamma_{\mu}(a,b)=\frac{a}{\mu}\frac{b}{\mu}\mu-\left(\int_M \frac{a}{\mu}\frac{b}{\mu}\mu\right)\mu.
\end{equation*}
To see that these really are the Christoffel symbols it remains to note that $\Gamma_{\mu}(a,b)\in T_{\mu}\Prob$, which follows since they integrate to zero. 

It remains to prove that $\nabar$ is non-metric for $\alpha\notin \{-1,0,1\}$. Let us assume that there exists a Riemannian metric $\bar G^{\alpha}$ on $\Prob$ such that $\nabar$ is compatible with $\bar G^\alpha$, i.e., 
\begin{equation}\label{compatibility}
D\left(\bar G^\alpha(b,c)\right).a-\bar G^\alpha(\nabar_ab,c)-\bar G^\alpha(b,\nabar_ac)=0,
\end{equation}
for all $a,b,c$ vector fields on $\Prob$ (see, e.g., \cite{do1992riemannian}). Then a necessary condition is
$$\bar G^\alpha(R^{(\alpha)}(a,b)c,c)=0,$$
where $R^{(\alpha)}$ is the Riemann curvature tensor of $\nabar$ (see, e.g., \cite{eisenhart1922riemann}). From Lemma~\ref{prop:alpha-curvature}, this condition can be rewritten
$$\GFR_{\mu}(b,c)\bar G^\alpha_\mu(a,c)=\GFR_{\mu}(a,c)\bar G^\alpha_\mu(b,c).$$
Fixing $\mu$ and $c_\mu$, we obtain that the mapping $a_\mu\mapsto \bar G^\alpha_\mu(a_\mu,c_\mu)/\GFR_\mu(a_\mu,c_\mu)$ is constant, and by symmetry of the metrics, that the quotient depends only on $\mu$, implying that
\begin{equation}\label{conformal}
\bar G_\mu^\alpha(a,b)=h(\mu)G_\mu(a,b),
\end{equation}
for some conformal factor $h(\mu)=e^{\theta(\mu)}$. The terms in the compatibility condition \eqref{compatibility} can then be computed using \eqref{conformal}, Lemma~\ref{lem:deriv} and the definition of $\nabar$: 
\begin{align*}
D\left(\bar G^\alpha(b,c)\right).a &=\GFR_{\mu}(b,c)Dh.a+h(\mu)\left(\GFR_{\mu}(Db.a,c)+\GFR_{\mu}(Dc.a,b)-\int\frac{a}{\mu}\frac{b}{\mu}\frac{c}{\mu}\mu\right),\\
\bar G^\alpha(\nabar_ab,c)&=h(\mu)\left(\GFR_{\mu}(Db.a,c)-\tau\int\frac{a}{\mu}\frac{b}{\mu}\frac{c}{\mu}\mu\right),\\
\bar G^\alpha(\nabar_ac,b)&=h(\mu)\left(\GFR_{\mu}(Dc.a,b)-\tau\int\frac{a}{\mu}\frac{b}{\mu}\frac{c}{\mu}\mu\right),
\end{align*}
where $\tau = \frac{1+\alpha}{2}$.
Using the fact that $Dh.a=h(\mu)D\theta.a$, we obtain that equation \eqref{compatibility} is written
$$\GFR_{\mu}(b,c)D\theta.a+\left(2\tau -1\right)\int\frac{a}{\mu}\frac{b}{\mu}\frac{c}{\mu}\mu=0,\quad\forall a,b,c\in T\Prob.$$
Now we show that for every $a\in T\Prob$, there exist $b,c\in T\Prob$ such that the second term of the LHS in the above equation vanishes while $\GFR_{\mu}(b,c)$ is not zero. This will imply that $\theta$ is constant, so that $\bar G^\alpha$ is in fact a constant multiple of the Fisher--Rao metric $\GFR$, yielding a contradiction for all values of $\alpha$ except for $\alpha=0$ (i.e., $\tau=1/2$), the only case where $\GFR$ is compatible with the $\alpha$-connection.

To show this last step, fix $a\in T\Prob$ and consider $b\in T\Prob$ that is also in the orthogonal of $a$ with respect to $\GFR$. Then $\frac{a}{\mu}b$ belongs to $T\Prob$. Denote
$$c=b-\GFR_\mu\left(b,\frac{a}{\mu}b\right)\GFR_\mu\left(\frac{a}{\mu}b,\frac{a}{\mu}b\right)^{-1}\frac{a}{\mu}b.$$
Then $c$ also belongs to $T\Prob$ and is, by construction, $\GFR$-orthogonal to $\frac{a}{\mu}b$, so that 
$$\int\frac{a}{\mu}\frac{b}{\mu}\frac{c}{\mu}\mu=0.$$
On the other hand
$$\GFR_{\mu}(b,c)=\GFR_{\mu}(b,b)-\GFR_{\mu}\left(b,\frac{a}{\mu}b\right)^2\GFR_{\mu}\left(\frac{a}{\mu}b,\frac{a}{\mu}b\right)^{-1}$$
is strictly positive since $\frac{a}{\mu}$ is not constant, because $a$ must integrate to zero.
\end{proof}

\subsection{The geometry of the $\alpha$-connections on $\Prob$}\label{sec:geometry_alpha}

In this section we aim to use the map $\Phi_{\alpha}$ as defined in \eqref{eq:Phi} for understanding the geometry of the $\alpha$-connection $\overline{\nabla}^{(\alpha)}$ on $\Prob$; for $\alpha\in (-1,1)$ similar calculations were done in~\cite{bauer2024p}; here we note how they extend to any $\alpha\in \R$.

The main idea in this section is to pull back an appropriate natural connection on a submanifold of $C^\infty(M)$. For this connection, the geodesic equation has a simpler structure than for $\overline{\nabla}^{(\alpha)}$, and we can obtain a closed formula for its solution, up to a solution of an ODE. In the following we will exclude the case $\alpha=1$, which is slightly different and was treated in Remark~\ref{rem:alpha1}.
\begin{lemma}
    Let $\alpha\ne 1$ and let $\Phi_{\alpha}$ as in \eqref{eq:Phi}.
    Then
    \[
    S_\alpha:=\Phi_{\alpha}(\Prob) = \left\{ f\in C^\infty(M) ~:~ f>0\text{ and } \, \int_M f^{\frac{2}{1-\alpha}}\, \lambda = \left(\frac{2}{|1-\alpha|}\right)^{\frac{2}{1-\alpha}}\right\}
    \]
    is a smooth submanifold of $C^\infty(M)$, 
    whose tangent space is
    \[
    T_f S_\alpha = \left\{ \xi \in C^\infty(M) ~:~ \int_M f^{\frac{1+\alpha}{1-\alpha}}\xi \,\lambda = 0 \right\}.
    \]
    Furthermore, the inverse $\Phi^{-1}_{\alpha}:S_\alpha\to\Prob $ is given by
    \begin{equation*}
    \Phi^{-1}_{\alpha}(f)=\left(\frac{|1-\alpha|}{2}f\right)^{\frac{2}{1-\alpha}}.
    \end{equation*}
\end{lemma}
\begin{proof}
This follows immediately from the formula for $\Phi_{\alpha}$: the first condition follows from the fact that elements of $\Prob$ are positive densities, whereas the second condition follows from the total mass constraint.  The remaining statements directly follow.
\end{proof}
\begin{definition}[Radial projection and the $\alpha$-connection on $S_\alpha$]
    Let $\alpha\neq 1$. The radial ($\alpha$-)projection $\pi^\alpha : TC^\infty(M)|_{S_\alpha} \to TS_\alpha$ is the projection with respect to the splitting
    $ T_fC^\infty(M) = T_fS_\alpha \oplus \operatorname{span}\{f\}$, that is,
    \[
    \pi_f^\alpha(\xi) = \xi - \left(\frac{|1-\alpha|}{2}\right)^{\frac{2}{1-\alpha}}\left(\int_M f^{\frac{1+\alpha}{1-\alpha}} \xi \,\lambda \right) f.
    \]
    This notion allows us define a new $\alpha$-connection on $S_\alpha$: 
    let $\nabla$ be the trivial connection on $C^\infty(M)$.
    We then define the $\alpha$-connection $\tilde{\nabla}^\alpha$ on $S_\alpha$ by
    \[
    (\tilde{\nabla}^\alpha_\xi \eta)_f = \pi^\alpha_f (\nabla_\xi \eta).
    \]
\end{definition}
The following proposition shows that the pullback of the $\alpha$-connection $\tilde{\nabla}^\alpha$ yields a multiple of the Amari-\u{C}encov $\alpha$-connection on $\Prob$:
\begin{proposition}
    Let $\alpha\neq 1$. Up to a constant depending on the footpoint, the pullback of $\tilde{\nabla}^\alpha$ to $\Prob$ via the map $\Phi_{\alpha}$ is the Amari-\u{C}encov $\alpha$-connection $\overline{\nabla}^{(\alpha)}$; more precisely we have
    \[
    \Phi^* \tilde{\nabla}^\alpha |_\mu = \operatorname{sgn}(1-\alpha)\left(\frac{\mu}{\lambda}\right)^{-\frac{1+\alpha}{2}} \overline{\nabla}^{(\alpha)}|_\mu.
    \]
    In particular, the geodesic equations of $\Phi^* \tilde{\nabla}^\alpha$ and $\overline{\nabla}^{(\alpha)}$ are the same.
\end{proposition}
\begin{proof}
The proof follows by direct calculation,  cf.~\cite[Theorem 3.12 and Theorem~4.10]{bauer2024p} where the case $\alpha\in(-1,1)$ was treated.
\end{proof}

By definition the geodesic equation of the connection $\tilde{\nabla}^\alpha$ is given by
$\pi^\alpha_\gamma (\ddot\gamma) = 0$, that is, $\ddot\gamma \in \operatorname{span}(\gamma)$, i.e., the geodesic equations can be written as
    \[
    \begin{cases}
        \ddot \gamma \,\parallel \, \gamma \\
        \int_M \gamma^{\frac{2}{1-\alpha}} \,\lambda = \left(\frac{2}{|1-\alpha|}\right)^{\frac{2}{1-\alpha}}.
    \end{cases}
    \]
Given the simplicity of the geodesic equation of $\tilde{\nabla}^\alpha$, one can guess that its solutions are projections of straight lines to $S_\alpha$.
For $\alpha\in (-1,1)$ this has been shown to be true, up to time reparametrization, cf.~\cite[Theorem 4.11]{bauer2024p}.
The following theorem states that this is true for any $\alpha\neq 1$, and extends this to treat both the initial and boundary value geodesic problems:
\begin{theorem}
    \label{thm:sulotion_prob}
    Let $\alpha\neq 1$. Let $f\in S_\alpha$ and let $\xi \in C^\infty(M)$, $\xi\ne f$.
    Let $I \subset \mathbb{R}$ be an interval containing $0$, and suppose $\tau : I \to \mathbb{R}$ satisfies the ODE
    \[
    \begin{split}
        \ddot\tau (t) &= 2 \frac{\int_M (f + \tau(t)\xi)^{\frac{1+\alpha}{1-\alpha}} \xi \,\lambda}{\int_M (f + \tau(t)\xi)^{\frac{2}{1-\alpha}} \,\lambda}\dot\tau(t)^2, \\
        \tau(0) &=0, \\
        \dot{\tau}(0) &=1.
    \end{split}
    \]
    Then $\gamma : I \to S_\alpha$, defined by
    \[
    \gamma(t) = \frac{2}{|1-\alpha|}\frac{f+ \tau(t)\xi}{\left(\int_M (f+\tau(t)\xi)^{\frac{2}{1-\alpha}}\,\lambda\right)^{\frac{1-\alpha}{2}}},
    \]
    is a geodesic of $\tilde{\nabla}^\alpha$ emanating from $\gamma(0)=f$.
    For $\xi\in T_f S_\alpha$, this is the unique geodesic with initial conditions $\gamma(0) = f$, $\dot{\gamma}(0) = \xi$.
    For $\bar{f}\in S_\alpha$ and $\xi = \bar{f} - f$, this is the unique geodesic from $f$ to $\bar{f}$.
\end{theorem}

\begin{corollary}[Geodesics of $\alpha$-connections on $\Prob$]
The geodesics of $\overline{\nabla}^{(\alpha)}$ on $\Prob$ are obtained by pulling back the geodesics of $\tilde{\nabla}^{\alpha}$ via the map $\Phi_\alpha$.
In particular, this gives the general solution of the geodesic equation, as well as showing that $(\Prob,\overline{\nabla}^{(\alpha)})$ is uniquely-geodesically-convex --- any two probability densities can be connected by a unique $\alpha$-geodesic.\footnote{We note that for finite dimensional parametric models, a similar equation for the geodesic connecting two probability measures is given in \cite[Theorems 14.1, 15.1]{morozova1991markov} (see also \cite[P.~51]{ay2017information}).
Yet, to best of our knowledge, this formula appears without a proof that $\tau$ indeed reaches 1, nor with a geometric explanation or an explicit ODE for $\tau$.}
\end{corollary}

\begin{proof}
The fact that $\gamma(t)$ is a geodesic is identical to the proof of \cite[Theorem 4.11]{bauer2024p}.
A straightforward calculation shows that 
\[
\dot\gamma(0) = \xi -\left(\frac{2}{|1-\alpha|}\right)^{-\frac{2}{1-\alpha}}\left(\int_M f^{\frac{1+\alpha}{1-\alpha}}\xi\,\lambda\right) f = \pi^\alpha_f(\xi),
\]
and in particular $\dot\gamma(0) = \xi$ if $\xi \in T_f S_\alpha$, as the second addend vanishes.
The fact that, in this case, $\gamma(t)$ is the unique geodesic with these initial conditions follows from the local well-posedness of $\overline{\nabla}^{(\alpha)}$ on $\Prob$ \cite[Theorem~4.2]{bauer2024p}.
This establishes that all the geodesics are obtained by projecting straight lines onto $S_\alpha$.
Thus, a geodesic between $f,\bar{f}\in S_\alpha$, if it exists, is unique, and given by $\gamma(t)$ as above with $\xi = \bar{f}-f$.
In order to prove its existence, we need to show that for some $t=T$ we have $\tau(T)=1$.
Note that $\tau$ must be monotonically increasing: otherwise, if $t_0$ is a local maximum of $\tau$, we immediately have that $\dot{\gamma}(t_0)=0$, in contradiction to the local well-posedness of the equation. Assume by contradiction that the geodesic equation with $\xi= \bar{f}-f$ never reaches $\tau=1$. Then we have $\lim_{t\to \infty}\tau(t) = t_0 \in (0,1]$.
Since $f + \tau\xi = (1-\tau)f + \tau \bar{f}$ is a positive function for all $\tau\in [0,1]$, we obtain that $\gamma(t)$ is well defined for all $t\in [0,\infty)$, and, as $t\to \infty$, we have
\[
\gamma(t) \to \gamma_\infty:=\frac{2}{|1-\alpha|}\frac{f+ \tau_0\xi}{\left(\int_M (f+\tau_0\xi)^{\frac{2}{1-\alpha}}\,\lambda\right)^{\frac{1-\alpha}{2}}}\in S_\alpha
\]
and 
\[
\dot\gamma(t) \to 0.
\]
Consider now the geodesic emanating from $\gamma_\infty$ with $\xi\in f- \gamma_\infty$.
This geodesic has non-zero initial velocity, exists for some time, and it coincides with the original geodesic, as both of them lie in the intersection of $S_\alpha$ with the two-dimensional plane spanned by $f$ and $\bar{f}$.
This is a contradiction to the local well-posedness of the equation, proving that $\lim_{t\to T_{\text{max}}}\tau(t) >1$, hence the geodesic from $f$ to $\bar{f}$ exists.
\end{proof}

\begin{remark}[Longtime existence of $\alpha$-geodesics and relation to the periodic generalized Proudman--Johnson equations.]
It follows from the above discussion that geodesics can blowup in one of two cases: either if $\tau\to \infty$ in finite time, or if $f+\tau \xi$ leaves the positive quadrant in finite time (here we use the fact that $M$ is compact, and so $(f+\tau \xi)^{2/(1-\alpha)}$ is integrable as long as $f+\tau \xi >0$ everywhere).
This suggests that the long time behavior of geodesics depends on the geometry of the manifold $S_\alpha$; for example, in the case $\alpha\in (-1,1)$, it was shown in \cite[Theorem 4.11]{bauer2024p}, using the above formula for geodesics, that blowup occurs for any initial condition, using the fact that in this case $S_\alpha$ is a convex (it is part of the $L^{\frac{2}{1-\alpha}}$-sphere), incomplete sub-manifold (it is shown that $\tau>t$ for $t>0$, hence for every $\xi\in T_fS_\alpha$ one hits the boundary of $S_\alpha$ in finite time).\footnote{We note a slight error in \cite{bauer2024p}:
the geodesic equation can be alternatively also written as
\[
    \ddot\tau (t) = 2 \left(1-\frac{\int_M (f + \tau(t)\xi)^{\frac{1+\alpha}{1-\alpha}} f \,\lambda}{\int_M (f + \tau(t)\xi)^{\frac{2}{1-\alpha}} \,\lambda}\right)\frac{\dot\tau(t)^2}{\tau}.
\]
If $\alpha\in (-1,1)$, one can use H\"older inequality and obtain that, for $f\equiv \frac{2}{|1-\alpha|}$,
\[
\ddot\tau (t) \ge 2 \left(1-\frac{\frac{2}{|1-\alpha|}}{\|\frac{2}{|1-\alpha|} + \tau(t)\xi)\|_{L^{\frac{2}{1-\alpha}}(\lambda)}}\right)\frac{\dot\tau(t)^2}{\tau},
\]
and not as written in \cite[Eq.~(25)]{bauer2024p}.
This does not affect the bounds in the rest of the analysis.}
The analysis of \cite[Theorem~4.11]{bauer2024p} extends to  $\alpha=-1$, where the lack of strict convexity of $S_\alpha$ results in the inequality $\tau\ge t$, again resulting in blowup as one reaches the boundary of $S_\alpha$ in finite time.
For the case $\alpha=1$, the space $\tilde{S}_1=\tilde{\Phi}_1(\Prob)$ is a flat, complete, linear space, hence geodesics exists for all time, cf.\ Remark~\ref{rem:alpha1}.

In Section~\ref{sec:PJnonperiodic} we described a relation between the non-periodic gPJ equations and the the $\alpha$-connections (the $\alpha$-FR metric, resp.) on $\DiffRshift$. 
Similarly, the periodic gPJ equations correspond to geodesics of the $\alpha$-connections on $\operatorname{Prob}(S^1)$ \cite{lenells2014amari}. In this case there exists a full characterization of blowup depending on $\alpha$ for these equations: namely,  the gPJ with parameter $\alpha$ is globally well-posed on $C^{\infty}(S^1)$ if and only if $\alpha\in [1,3)$, cf.~\cite{kogelbauer2020global,sarria2013blow}. 
The proofs for these results rely on the transformation $\Phi_{\alpha}$ as defined in Theorem~\ref{thm:explicit_dens}; this section establishes the geometric meaning and importance of them. It would be interesting if the geometric interpretation of the present article could shed further light on these blowup results; in particular, one could try to relate the existence of blow-up to geometric properties of $S_\alpha$ like convexity/concavity and completeness/incompleteness. Additionally, we expect that Theorem~\ref{thm:sulotion_prob} will enable to generalize these blowup results to any closed $M$.
\end{remark}

The remaining case to consider is $\alpha=1$. We have already seen in Remark~\ref{rem:alpha1} that in that case geodesics are pullback of straight lines and exist for all time. One can find their explicit expression similarly as the finite-dimensional, parametric case, see~\cite{ay2017information}:
\begin{proposition}[Case $\alpha=1$]\label{prop:alpha1}
The $\overline\nabla^{(1)}$-geodesic $\gamma$ on $\Prob$ with initial conditions $\gamma(0)=\mu$, $\dot\gamma(0)=a$ is given by
$$\gamma(t)=\frac{e^{t\frac{a}{\mu}}}{\int_M e^{t\frac{a}{\mu}}\mu}\mu,$$
and the logarithm map is
$$\log^{(1)}_{\mu}\nu=\left(\log\frac{\nu}{\mu}+\KL(\mu|\nu)\right)\mu,$$
where $\KL$ denotes the Kullback-Leibler divergence. 
\end{proposition}
Just like in the finite dimensional case, this allows to show the following Pythagorean relation: 
\begin{corollary}[Amari's Pythagorean relation]
If $\mu,\nu,\omega\in \Prob$ be probability densities such that the $\nabla^{(-1)}$-geodesic connecting $\mu$ and $\nu$ is $G$-orthogonal to the $\nabla^{(1)}$-geodesic connecting $\nu$ and $\omega$ at $\nu$, then
$$\KL(\mu|\omega)=\KL(\mu|\nu)+\KL(\nu,\omega).$$
\end{corollary}

\subsection{Finite dimensional statistical models}

In this final section, we investigate the Riemannian property of the $\alpha$-connections in the case of finite-dimensional submanifolds of $\operatorname{Prob}(\R^n)$ that describe parametric statistical models. Consider a parametric family of probability densities $\mathcal P_\Theta = \{p(x,\theta)\,dx: \theta\in\Theta\}$, where $x\in\R^n$ denotes the sample variable, $dx$ denotes the Lebesgue measure on $\R^n$ and $\theta\in\R^d$ is the parameter. 
The Fisher--Rao metric on $\mathcal{P}_\Theta$ is obtained by restricting the metric from $\operatorname{Prob}(\R^n)$, and the $\alpha$-connections are obtained by projecting the connections from the larger space.
Below we recall the explicit expressions, when  $\mathcal{P}_\Theta$ is identified with the parameter space $\Theta$.
\begin{definition}[Fisher--Rao metric and $\alpha$-connections on $\Theta$]
    The Fisher--Rao metric on the parameter space $\Theta$ is defined for any $1\leq i,j\leq d$ by
    \[
    G_{ij}(\theta):=E_\theta\left[\partial_i\ell(X,\theta)\partial_j\ell(X,\theta)\right]
    \]
    where $\partial_i:=\frac{\partial}{\partial\theta^i}$, $\ell(x,\theta)=\log p(x,\theta)$ is the log-density and $E_\theta$ denotes expectation with respect to the density $p(\cdot,\theta)$. The $\alpha$-connection $\nabla^{(\alpha)}$ has Christoffel symbols defined for any $1\leq i,j,k\leq d$ by
    $$\Gamma^{(\alpha)}_{ij,k}(\theta):=G_\theta(\nabla^{(\alpha)}_{\partial_i}\partial_j,\partial_k):=E_\theta\left[\left(\partial_i\partial_j\ell(X,\theta)+\frac{1-\alpha}{2}\partial_i\ell(X,\theta)\partial_j\ell(X,\theta)\right)\partial_k\ell(X,\theta)\right].$$
\end{definition}
Note that here and below, the Christoffel symbols of the first kind $\Gamma_{ij,k}$ are obtained from the Christoffel symbols of the second kind $\Gamma_{ij}^k$ by lowering the index with respect to the Fisher--Rao metric, even when the connection is not metric with respect to it.

The following example shows that in finite dimensions, we can get a different behavior than the one described in Theorem~\ref{thm:alpha_nonmetric} in infinite dimensions. For this particular model, taken from \cite{amari2000methods}, all $\alpha$-connections are flat and thus Riemannian.
\begin{example}\label{ex:finite-dim}
Consider a probability measure $q$ on $\R$, $d$ i.i.d. random variables $Y_1,\hdots,Y_d\sim q$, and the distribution of the random vector $X=Y+\theta$ obtained by translating $Y=(Y_1,\hdots,Y_d)$ in the direction of a given $\theta\in\R^d$, i.e.,
$$\mathcal M=\{p(x,\theta)=q^{(d)}(x-\theta),\theta\in\R^d\}\quad \text{where}\quad q^{(d)}(x)=\prod_{i=1}^dq(x_i) \quad \forall x=(x_1,\hdots,x_d).$$
Then the Fisher--Rao metric is given by
\begin{align*}
G_{ij}(\theta)&=E_{X\sim p(\cdot,\theta)}\left[(\log q)'(X_i-\theta_i)(\log q)'(X_j-\theta_j)\right]\\
&=\delta_{ij}r \quad \text{where} \quad r=E_{Y\sim q}\left[(\log q)'(Y)^2\right],
\end{align*}
where we have used the fact that $E_q\left[(\log q)'(Y)\right]=0$. Notice that $r$ is independent of $\theta$ and thus the Fisher--Rao metric on $\mathcal M$ is Euclidean. Similarly, the Christoffel symbols of the $\alpha$-connection are given by
\begin{align*}
\Gamma_{ij,k}^{(\alpha)}(\theta)&=-E_{Y \sim q^{(d)}}\left[\left(\delta_{ij}(\log q)''(Y_i)+\frac{1-\alpha}{2}(\log q)'(Y_i)(\log q)'(Y_j)\right)(\log q)'(Y_k)\right]\\
&=\begin{cases}
-E_{Y\sim q}\left[\left((\log q)''(Y)+\frac{1-\alpha}{2}(\log q)'(Y)^2\right)(\log q)'(Y)\right] \quad \text{if}\quad i=j=k\\
0 \quad \text{otherwise}.
\end{cases}
\end{align*}
Thus $\Gamma_{ij,k}^{(\alpha)}$ is zero unless $i=j=k$, for which it is equal to a constant that does not depend on $\theta$. Dropping the $\alpha$ exponent in the Christoffel symbols for the sake of simplicity, we obtain for the curvature tensor
\begin{align*}
{R^{(\alpha)}}_{ijk}^\ell&=\frac{\partial}{\partial\theta_i}\Gamma_{jk}^\ell -\frac{\partial}{\partial\theta_j}\Gamma_{ik}^\ell+\Gamma_{im}^\ell\Gamma_{jk}^m-\Gamma_{jm}^\ell\Gamma_{ik}^m\\
&=\frac{1}{r^2}\left(\Gamma_{im,\ell}\Gamma_{jk,m}-\Gamma_{jm,\ell}\Gamma_{ik,m}\right)=0,
\end{align*}
so that the $\alpha$-connection is flat on the model $\mathcal M$ identified with its parameter space $\R^d$, for any value of $\alpha$. Thus it admits a local parallel frame field, which is global since the manifold is simply connected. By declaring this frame field to be orthonormal one defines a metric with respect to which the connection is metric. Since $\nabla^{(\alpha)}$ is torsionless, it is the Levi-Civita connection of this metric. Thus the $\alpha$-connections on $\mathcal M$ are Riemannian connections for all $\alpha$, unlike in the infinite-dimensional setting.
\end{example}

By contrast, for other standard examples, namely exponential families with a two-dimensional parameter space, we get the same picture as in the non parametric setting: the $\alpha$-connections are non metric except for $\alpha\in\{\pm 1,0\}$. Note that the case $\alpha=\pm 1$ is not trivial since the definition of the $\alpha$-connections on these statistical models are obtained by projecting the connections from the larger space by the FR metric, which is not the metric with respect to which the connection is metric. 
\begin{definition}[Exponential family]
    A parametric model is an exponential family with natural parameter $\theta$ if the density is of the form
    \begin{equation}\label{exp_family}
    p(x,\theta)=\exp(h(x)+\theta\cdot\varphi(x)-F(\theta))
    \end{equation}
    for certain functions $h$, $\varphi$ on $\R^n$ and $F$ on $\Theta$.
\end{definition}
The following geometric properties of exponential families are well-known. 
\begin{lemma}[Geometry of exponential families]\label{lem:exponential}
    Let $\mathcal P_\Theta = \{p(x,\theta)\,dx: \theta\in\Theta\}$ be an exponential family, i.e., such that $p$ is of the form \eqref{exp_family}. Then for $1\leq i,j,k,\ell\leq d$:
    \begin{itemize}
    \item The Fisher--Rao metric is the hessian of the log-partition function $F$
    \begin{equation}\label{exp_fisher_rao}
    G_{ij}(\theta)=\partial_i\partial_j F(\theta).
    \end{equation}
    \item The Christoffel symbols of the $\alpha$-connections are given by
    \begin{equation}\label{exp_alpha}
    \Gamma^{(\alpha)}_{ij,k}=\frac{1-\alpha}{2}F_{ijk}.
    \end{equation}
    \item The curvature tensor is written
    \begin{equation}\label{exp_alpha_curv}{R^{(\alpha)}}_{ijk}^\ell=\frac{1-\alpha^2}{4}G^{rm}G^{s\ell}\left(F_{ikm}F_{jrs}-F_{jkm}F_{irs}\right)
    \end{equation}
    \end{itemize}
    Here we have used the notation $F_{ijk}:=\partial_i\partial_j\partial_kF$.
\end{lemma}
\begin{proof}
The expression for $G_{ij}$ follows from straightforward computation. Since the hessian of the log-density is independent of the sample variable and $\partial_k\ell(X,\theta)$ has zero expectation, we get
\begin{align*}
\Gamma^{(\alpha)}_{ij,k}&=-\partial_i\partial_jF(\theta)E_\theta\big[\partial_k\ell(X,\theta)\big]+\frac{1-\alpha}{2}E_\theta\big[\partial_i\ell(X,\theta)\partial_j\ell(X,\theta)\partial_k\ell(X,\theta)\big]\\
&=\frac{1-\alpha}{2}E_\theta\big[\partial_i\ell(X,\theta)\partial_j\ell(X,\theta)\partial_k\ell(X,\theta)\big].
\end{align*}
Denoting $\partial_if=f_i$, we notice that, since $\ell_{ij}=f_{ij}/f- f_{i}f_j/f^2$ and $\ell_{ij}$ is constant in the sample variable,
\begin{align*}
E_\theta\big[\ell_{ijk}\big]&=E_\theta\left[\partial_i\left(\frac{f_{jk}}{f}-\frac{f_jf_k}{f^2}\right)\right]=E_\theta\left[\frac{f_{ijk}}{f}-\frac{f_{jk}f_i}{f^2}-\frac{f_{ki}f_j}{f^2}-\frac{f_{ij}f_k}{f^2}+2\frac{f_if_jf_k}{f^3}\right]\\
&=0-E_\theta\big[\ell_{jk}\ell_i\big]-E_\theta\big[\ell_{ki}\ell_j\big]-E_\theta\big[\ell_{ij}\ell_k\big]-E_\theta\big[\ell_i\ell_j\ell_j\big]\\
&=-\ell_{jk}E_\theta\big[\ell_i\big]-\ell_{ki}E_\theta\big[\ell_j\big]-\ell_{ij}E_\theta\big[\ell_k\big]-E_\theta\big[\ell_i\ell_j\ell_j\big]=-E_\theta\big[\ell_i\ell_j\ell_j\big],
\end{align*}
and since $E_\theta\big[\ell_{ijk}\big]=-F_{ijk}$ we obtain $\Gamma_{ij,k}^{(\alpha)}=\frac{1-\alpha}{2}F_{ijm}$. Then the curvature tensor defined by $R^{(\alpha)}(\partial_i,\partial_j)\partial_k=:{R^{(\alpha)}}_{ijk}^\ell \partial_\ell$ is written
\begin{align*}
{R^{(\alpha)}}_{ijk}^\ell&=\partial_i\Gamma_{jk}^\ell -\partial_j\Gamma_{ik}^\ell+\Gamma_{im}^\ell\Gamma_{jk}^m-\Gamma_{jm}^\ell\Gamma_{ik}^m\\
&=\frac{1-\alpha}{2}\left(F_{ijkm}G^{m\ell}+F_{jkm}\partial_iG^{m\ell}-F_{ijkm}G^{m\ell}+F_{ikm}\partial_jG^{m\ell}\right)\\
&+\frac{(1-\alpha)^2}{4}\left(F_{jkr}g^{rm}F_{ims}G^{s\ell}-F_{ikr}G^{rm}F_{jms}G^{s\ell}\right).
\end{align*}
Then using the fact that $\partial_iG^{mr}G_{rs}=-G^{mr}\partial_iG_{rs}$ and so $\partial_iG^{m\ell}=-G^{mr}\partial_iG_{rs}G^{s\ell}$, we obtain
\begin{align*}
{R^{(\alpha)}}_{ijk}^\ell&=\frac{1-\alpha}{2}\left(F_{jkm}F_{irs}G^{rm}G^{s\ell}+F_{ikm}F_{jrs}G^{mr}G^{s\ell}\right)\\&\qquad+\frac{(1-\alpha)^2}{4}\left(F_{jkr}F_{ims}G^{rm}G^{s\ell}-F_{ikr}F_{jms}G^{rm}G^{s\ell}\right)\\
&=\frac{1-\alpha}{4}\bigg(-2F_{jkm}F_{irs}G^{rm}G^{s\ell}+F_{ikm}F_{jrs}G^{rm}G^{s\ell}\\&\qquad+(1-\alpha)F_{jkm}F_{irs}G^{rm}G^{s\ell}-(1-\alpha)F_{ikm}F_{jrs}G^{rm}G^{s\ell}\bigg)\\
&=\frac{1-\alpha^2}{4}G^{rm}G^{s\ell}\left(F_{ikm}F_{jrs}-F_{jkm}F_{irs}\right)
\end{align*}
as expected.
\end{proof}
The following lemma summarizes the Riemannian or non-Riemannian properties of $\na$ on a two-dimensional exponential family. 
\begin{proposition}\label{prop:exponential}
    Let $\mathcal P_\Theta = \{p(x,\theta)\,dx: \theta\in\Theta\}$ be an exponential family, i.e., such that $f$ is of the form \eqref{exp_family}. Assume that the Fisher--Rao metric on $\Theta$ is not flat and $\dim\Theta=2$. Then
    \begin{itemize}
    \item If $\alpha=1$, then $\nabla^{(\alpha)}$ is the Levi-Civita connection of the Euclidean metric on $\Theta$.
    \item If $\alpha=0$, then $\nabla^{(\alpha)}$ is the Levi-Civita connection of the Fisher--Rao metric.
    \item If $\alpha=-1$, then $\nabla^{(\alpha)}$ is the Levi-Civita connection of the metric $G^{(-1)}(\theta)=(\mathrm{Hess}_\theta F)^2$.
    \item If $\alpha \notin\{1,0,-1\}$ then $\nabla^{(\alpha)}$ is a non metric affine connection.
    \end{itemize}
\end{proposition}
\begin{remark}
Note that in the infinite-dimensional setting, it is the $(-1)$-connection that is the trivial connection while here for finite-dimensional submanifolds corresponding to exponential models, it is the $1$-connection. On the finite-dimensional version of $\Prob$ defined by the submanifold of probability measures on a finite set (a mixture model, see~\cite{amari2000methods}), it is the $(-1)$-connection that is trivial, just like in the infinite-dimensional setting. This is the reason why the $1$ and $(-1)$-connections are called the exponential and mixture connections, respectively.
\end{remark}

\begin{proof}[Proof of Proposition~\ref{prop:exponential}]
If $\alpha=1$, $\na$ is the trivial connection on $\Theta$ associated with the Euclidean metric on $\Theta$, and if $\alpha=0$ it is well known that it corresponds to the Fisher--Rao metric \cite{amari2000methods}. If $\alpha=-1$, the curvature is also flat and $\eta=\nabla F(\theta)$ are affine coordinates for $\na$ \cite{amari2000methods}, so that pulling back the Euclidean metric yields
$$G^{(-1)}_\theta(u, v)=u^\top(\mathrm{Hess}_\theta F)^2 v, \quad \forall v\in\R^d.$$
It remains to consider the case $\alpha \notin\{1,0,-1\}$. If the connection $\nabla^{(\alpha)}$ was compatible with a Riemannian metric $g$, then for any vector fields $X,Y,Z,W$ we would have $g(R^{(\alpha)}(X,Y)Z,W)+g(R^{(\alpha)}(X,Y)W,Z)=0$, i.e., in local coordinates,
\[
\sum_m {R^{(\alpha)}}^m_{ijk}g_{m\ell} + \sum_m {R^{(\alpha)}}^m_{ij\ell} g_{mk} = 0 \qquad \text{for all $i$,$j$,$k$,$\ell$.}
\]
In dimension $d=2$, denoting $a={R^{(\alpha)}}_{121}^1$, $b={R^{(\alpha)}}_{121}^2$, $c={R^{(\alpha)}}_{122}^1$ and $d={R^{(\alpha)}}_{122}^2$, this yields
\begin{equation}\label{compatibility_cond}
\left(\begin{matrix} a & b & 0\\ -c & -(a+d) & -b\\ 0 & c & d\end{matrix}\right)\left(\begin{matrix} g_{11}\\ g_{12}\\ g_{22}\end{matrix}\right)=0.
\end{equation}
Rearranging formula~\eqref{exp_alpha_curv}, we find that the entries of the curvature tensor are given by
\begin{align*}
a=\frac{1-\alpha^2}{4}G^{12}S, \quad b=\frac{1-\alpha^2}{4}G^{22}S, \quad c=-\frac{1-\alpha^2}{4}G^{11}S, \quad d=-\frac{1-\alpha^2}{4}G^{12}S
\end{align*}
where
\begin{align*}
S=\frac{1}{\det G}\big(F_{22}(F_{111}F_{122}-F_{112}^2)+F_{11}(F_{222}F_{112}-F_{122}^2)-F_{12}(F_{111}F_{222}-F_{112}F_{122})\big).
\end{align*}
The matrix $M$ on the LHS of the compatibility condition~\eqref{compatibility_cond} has characteristic polynomial
$P(\lambda)=-\lambda^3+(a^2+d^2+ad-2bc)\lambda+(a+d)(bc-ad)$, where from the calculations above we have $a=-d$, so that
$$P(\lambda)=\lambda\big(a^2-2bc-\lambda^2\big).$$
Thus $0$ is an eigenvalue of $M$ and a metric $g$ that is compatible with $\na$, if it exists, is described by a corresponding eigenvector. Since we have assumed that $R^{(0)}$ is not identically zero, $S$ is not zero, and thus we have for $\alpha\notin\{-1,1\}$
\begin{align*}
a^2-2bc&=\left(\frac{1-\alpha^2}{4}\right)^2S^2\big((G^{12})^2+2G^{11}G^{22}\big)=\left(\frac{1-\alpha^2}{4}\right)^2S^2\left(3(G^{12})^2+\frac{2}{\det G}\right)>0,
\end{align*}
and the $0$-eigenspace is one-dimensional. 
Since it contains the Fisher--Rao metric $G$, it is the set of $g$ that are conformal to $G$
$$g(\theta)=h(\theta)G(\theta)$$
for some positive function $h$. Assume that such a metric $g$ is compatible with $\nabla^{(\alpha)}$. Then its Christoffel symbol is
$${\Gamma^{(\alpha)}}_{ij}^k={\Gamma^{(0)}}_{ij}^k+\frac{1}{2h}G^{mk}\left(\partial_ih\,G_{jm}+\partial_j h\,G_{mi}-\partial_mh\,G_{ij}\right).$$
Combining this with formulas \eqref{exp_fisher_rao} and \eqref{exp_alpha}, we obtain
$$\frac{1-\alpha}{2} F_{ijm}G^{mk}=\frac{1}{2}F_{ijm}G^{mk}+\frac{1}{2}G^{mk}\left(\frac{\partial_ih}{h}G_{jm}+\frac{\partial_j h}{h}G_{mi}-\frac{\partial_mh}{h}G_{ij}\right).$$
Finally, setting $\tilde h=\log h$,
$$-\alpha F_{ijm}=\partial_i\tilde h\, F_{jm} + \partial_j \tilde h\, F_{mi}-\partial_m \tilde h\, F_{ij}.$$
Since $\dim\Theta=2$, we obtain the following equations
$$-\frac{1}{\alpha}\partial_1\tilde h=\frac{F_{111}}{F_{11}}=\frac{F_{221}}{F_{22}}=\frac{F_{121}}{F_{12}}, \quad -\frac{1}{\alpha}\partial_2\tilde h=\frac{F_{112}}{F_{11}}=\frac{F_{222}}{F_{22}}=\frac{F_{122}}{F_{12}},$$
so that integrating with respect to one variable and the other yields
$$\begin{cases}F_{11}=k F_{12}\\ F_{22}=\tilde kF_{12}\end{cases}\Rightarrow 
\begin{cases}F_1=kF_2+k'\\ F_2=\tilde kF_1+\tilde k'\end{cases}\Rightarrow
\begin{cases}F_{12}=k F_{22}\\ F_{12}=\tilde k F_{12}\end{cases}\Rightarrow \tilde k=1/k,$$
and finally we obtain that the Fisher--Rao metric is of the form $G(\theta)=r(\theta)\left(\begin{matrix} k & 1\\ 1 & 1/k\end{matrix}\right)$ which is not invertible and yields a contradiction.
\end{proof}

\end{document}